\documentclass[9pt,shortpaper,twoside,web]{ieeecolor}
\usepackage{generic}
\usepackage{cite}
\usepackage{hyperref}
\hypersetup{colorlinks=false, hidelinks}
\usepackage{amsmath,amssymb,amsfonts,amsthm}
\usepackage{graphicx}
\usepackage{textcomp}

\usepackage{array}
\usepackage[caption=false,font=normalsize,labelfont=sf,textfont=sf]{subfig}
\usepackage{stfloats}
\usepackage{url}
\usepackage{verbatim}
\usepackage{}
\usepackage[normalem]{ulem}
\usepackage{soul}
\usepackage{xcolor}
\usepackage{physics}
\usepackage{mathtools}
\usepackage{mathrsfs}
\usepackage{algorithm}
\usepackage{algpseudocode}
\usepackage{pdfpages}

\theoremstyle{plain}%
\newtheorem{thm}{Theorem}
\newtheorem{lem}{Lemma}
\newtheorem{prop}{Proposition}

\newtheorem{asm}{Assumption}

\newcommand{\bunderline}[1]{\underline{#1\mkern-2mu}\mkern2mu}
\newcommand{\E}{\mathbb{E}}
\newcommand{\KL}{\text{KL}}

\newcommand{\expvalb}[2]{\E_{#1}\left[ #2 \right]}
\newcommand{\Chebyshevbasislow}{9}
\newcommand{\Chebyshevbasishigh}{\the\numexpr(\Chebyshevbasislow*2)\relax}
\newenvironment{update}{\par\color{black}}{\par}

\def\BibTeX{{\rm B\kern-.05em{\sc i\kern-.025em b}\kern-.08em
    T\kern-.1667em\lower.7ex\hbox{E}\kern-.125emX}}
\begin{document}

\title{Gaussian-Based Parametric Bijections For Automatic Projection Filters}

\author{Muhammad F. Emzir, Zheng Zhao, Lahouari Cheded, Simo S\"arkk\"a}

\maketitle

\begin{abstract}
The automatic projection filter is a recently developed numerical method for projection filtering that leverages sparse-grid integration and automatic differentiation. However, its accuracy is highly sensitive to the accuracy of the cumulant-generating function computed via the sparse-grid integration, which in turn is also sensitive to the choice of the bijection from the canonical hypercube to the state space. In this paper, we propose two new adaptive parametric bijections for the automatic projection filter. The first bijection relies on the minimization of Kullback--Leibler divergence, whereas the second method employs the sparse-grid Gauss--Hermite quadrature. The two new bijections allow the sparse-grid nodes to adaptively move within the high-density region of the state space, resulting in a substantially improved approximation while using only a small number of quadrature nodes. The practical applicability of the methodology is illustrated in three simulated nonlinear filtering problems.
\end{abstract}

\begin{IEEEkeywords}
projection filter, adaptive bijection, numerical quadrature, automatic differentiation, sparse-grid integration
\end{IEEEkeywords}

\section{Introduction}
\label{sec:introduction}

The projection filter \cite{hanzon1991,brigo1999} is an approximate optimal filtering method based on projections of probability densities on manifolds. Specifically, the filter projects the Kushner--Stratonovich equation \cite{Kushner:1964} of the optimal filtering solution onto a finite-dimensional manifold of parametric densities, producing a finite-dimensional stochastic differential equation (SDE) representation. %
When the manifold is the manifold formed by mixtures of probability densities and the $L^2$ metric is used, then the projection filter is equivalent to a Galerkin method for solving the Kushner--Stratonovich equation (see \cite[Theorem 5.1]{armstrong2016a}).

The projection filter has so far been used only in limited scenarios. Outside of the Gaussian family, the most prevalent applications have been in univariate dynamical systems \cite{brigo1999,armstrong2023,koyama2018,kutschireiter2022}. Recently, Emzir et al.\ \cite{emzir2023} introduced an automatic projection filter for a wider class of filtering problems. Although the method is applicable to a large class of multivariate dynamical systems, its computational can be demanding for using a large number of quadrature nodes \cite{emzir2023}. Using a large number of nodes is often needed because of the fixed (rather than an adaptive) bijection between the canonical hypercube and the integration domain. In the other way around, the fixed bijection also practically limits the numerical accuracy of the filter when the computational budget is limited. 

\begin{update}
    In this study, we propose two bijections that transform the nodes of sparse-grid quadratures to adaptively cover the high-density domain of the projected conditional densities. The main challenge in constructing such bijections is found to be the impossibility of obtaining a bijection that transforms the integrand into a polynomial function in the transformed space, thereby achieving zero integration error (see Proposition \ref{prop:bijection_polynomial}). Further, we show that using squared integration error as an optimization cost function also leads to non-explicit bijection functions (see Proposition \ref{prop:convexity_of_abs_En}). To overcome these difficulties, in the first bijection, we minimize the Kullback--Leibler divergence between the projected conditional density and a Gaussian density via moment matching (see Proposition \ref{prop:zeta_xi_gaussian}). Not only does this result in an explicit bijection form, but it also optimizes the squared integration error under certain technical conditions (see Proposition \ref{prop:optimality}). We then adapt this technique to Gauss--Hermite quadrature.
\end{update} 
When the automatic projection filter algorithm is combined with these bijections, the part of the state space which has a high filtering density can be automatically tracked. This improvement requires much fewer quadrature nodes than those using static bijections. We illustrate the effectiveness of the proposed bijections using three numerical examples where it is challenging to confine conditional densities inside a fixed domain. We also generalize the filtering problems in comparison to \cite{emzir2023} by only requiring that the drift and diffusion functions belong to a vector space that is closed under partial differentiation.

The paper is organized as follows. We first review the projection filter for the exponential family in Section \ref{sec:AutomaticProjectionFilter}, and then explain how to propagate the SDE parameters using numerical integration and automatic differentiation. Section \ref{sec:ParametricBijections} delivers the main contributions of the paper, which are the two new bijections from the canonical hypercube to the state space of the projection filter. Section \ref{sec:NumericalExamples} shows the practical applicability of the bijections in the automatic projection filter in three simulated nonlinear filtering problems. Finally, Section \ref{sec:Conclusions} summarizes the results and concludes the paper.

\section{Automatic Projection Filter}\label{sec:AutomaticProjectionFilter}

In this section, we review the relevant theoretical results that constitute the foundation of the automatic projection filter \cite{emzir2023} and also to present the extension of the model considered in this paper. We consider optimal filtering problems on the following state-space model consisting of continuous-time stochastic dynamic and observation models:
\begin{subequations}
	\label{eqs:nonlinear_SDE}
	\begin{align}
		dx_t &= f(x_t) \, dt + \varrho(x_t) \, dW_t,\\
		dy_t &= h(x_t) \, dt + dV_t,
	\end{align}
\end{subequations}
where $x_t \in \mathcal{X} \coloneqq \mathbb{R}^d, y_t\in \mathcal{Y} \coloneqq \mathbb{R}^{d_y}$. The processes $\{ W_t , t \geq 0\}$  and  $\{ V_t , t \geq 0\}$ are independent Wiener processes taking values in $\mathbb{R}^{d_w}$ and $\mathbb{R}^{d_y}$ with invertible spectral density matrices $Q_t$ and $R_t$ for all $t\geq0$, respectively. For the sake of exposition, and without a loss of generality, in the rest of the paper, we assume that $R_t = I$ for all $t\geq0$. 

The conditional probability density of the state at time $t$, space $x_t$, given a history of measurements $y_\tau$, $0\leq \tau \leq t$, satisfies the Kushner--Stratonovich equation~\cite{Kushner:1964}. Let us define a class of probability densities $\mathcal{P}$ with respect to the Lebesque measure on $\mathcal{X}$ as $\mathcal{P} = \{p\in L^1: \int_\mathcal{X} p(x) \, dx = 1, p(x)\geq 0, \forall x \in \mathcal{X}\}$. In particular, let us consider the exponential family
\begin{align}
    \mathrm{EM}(c) \coloneqq \left\{ p \in \mathcal{P} \colon p(x) = \exp(c(x)^\top\theta - \psi(\theta)) \right\},\label{eq:EM_c}
\end{align}
where $\theta \in \Theta \subset \mathbb{R}^m$ is the natural parameter and $c\colon\mathbb{R}^d\to \mathbb{R}^m$ is a vector of natural statistics that are assumed to be linearly independent. The natural parameter space $\Theta$ is defined as
\begin{align}
    \Theta \coloneqq \left\{ \theta \in \mathbb{R}^m\colon \int_\mathcal{X} \exp(c(x)^\top \theta) dx <\infty \right\}. \label{eq:Theta}
\end{align}
An exponential family is said to be regular if $\Theta$ is an open subset of $\mathbb{R}^m$. In this work, we focus on regular exponential families. In the development of the proposed improved projection filter, we extensively use the cumulant-generating function (i.e., the log Laplace transform or log partition function \cite{brown1986,emzir2023}) defined by
\begin{align}
    \psi(\theta) = \log \left[ \int_\mathcal{X} \exp(c(x)^\top \theta) dx \right], \quad \theta \in \Theta.
    \label{eq:cumulant-generating}
\end{align}

Because the exponential family is assumed to be regular and the natural statistics are linearly independent, the exponential family is minimal \cite{kass1997,calin2014}. We recall the following standard result for a minimal regular exponential family \cite[Theorems 2.2.1 and 2.2.5]{kass1997}.
\begin{thm}\label{thm:kass_exponential_family_expectation}
    In a regular exponential family, the set $\Theta$ as defined in \eqref{eq:Theta} is convex. The cumulant-generating function $\psi(\theta)$ is strictly convex on $\Theta$ and it is differentiable up to an arbitrary order. The moments of the natural statistics $c_i(x)$, $i=1,\ldots,m$ exist for any order, and the expectations of
     $c_i$ and the corresponding Fisher information matrix $g$ are, respectively, given by,
    \begin{align}
        \E_\theta\left[ c_i \right] &= \pdv{\psi(\theta)}{\theta_i}, &
        g_{i,j}(\theta) &= \pdv[2]{\psi(\theta)}{\theta_i}{\theta_j}.
    \end{align}
    If the representation is minimal, then $g$ is positive definite.
\end{thm}
Let us denote by $\mathcal{S} = \left\{ p_\theta: \theta \in \Theta \subseteq \mathbb{R}^{m} \right\}$ a class of parametric densities which does not have to be the exponential family. There are a few approaches that can be used to project the Kushner--Stratonovich equation onto the manifold of parametric densities \cite{brigo2022}. The standard way is to leverage the property that the square root of the density $\sqrt{p_t}$ belongs to $L^2$. Therefore, by requiring that the stochastic differential $d\sqrt{p_t}$ is an element of $L^2$ and  by expressing the Kushner--Stratonovich equation in its Stratonovich form, we can project $d\sqrt{p_t}$ onto the tangent space $T_{\sqrt{p_\theta}} \mathcal{S}^{1/2}$. This is elucidated in the following lemma \cite[Lemma 2.1]{brigo1999}.
\begin{lem}\label{lem:projection_to_S_half}
     Let $p_\theta \in \mathcal{S}$, and $u$ be a function such that $\mathbb{E}_{p_\theta}[\abs{u}^2]< \infty$. Then the projection of $v\coloneqq u \sqrt{p_\theta} \in L^2$ onto the tangent space $T_{\sqrt{p_\theta}} \mathcal{S}^{1/2}$ is given by
    \begin{align}
        \Pi_\theta v &= \sum_{i=1}^m \sum_{j=1}^m 4 g^{-1}_{ij} \expval{v, \frac{1}{2\sqrt{p_\theta}}\pdv{p_\theta}{\theta_j}}\frac{1}{2\sqrt{p_\theta}}\pdv{p_\theta}{\theta_i}, \label{eq:Projection_to_S_half}
    \end{align}    
    where the inner product is defined as $\expval{u,v} = \int u(x) \, v(x) \, dx$.  
\end{lem}

For the exponential family $\text{EM}(c)$, the projection filter using Stratonovich projection is given by \cite{brigo1999}:
\begin{equation}
	\begin{split}
		d\theta_t &= g(\theta_t)^{-1}\E_{\theta_t}\left[ \mathcal{L}\left[c\right]  - \frac{1}{2}h^\top h \left[ c - \eta(\theta_t) \right] \right] dt\\
		&\quad + g(\theta_t)^{-1} \sum_{k=1}^{d_y} \E_{\theta_t}\left[ h_k\left[ c-\eta(\theta_t) \right] \right] \circ dy_{t,k}.\label{eq:Brigo_dTheta_EM}
	\end{split}
\end{equation}
In the equation above, $\E_{\theta}$ is the expectation with respect to the parametric probability density $p_\theta$, $h_k$ is the $k$-th element of $h$, $\theta_t\mapsto\eta(\theta_t)\coloneqq\E_{\theta_t}[c]$, $\mathcal{L}$ is the backward Kolmogorov diffusion operator, and
$\circ$ denotes the Stratonovich multiplication. 

We can now extend the class of state-space models considered in \cite{emzir2023} as follows. Let $\mathscr{P}$ be a set of smooth functions $\varphi\colon\mathbb{R}^d\to \mathbb{R}$ such that $\mathscr{P}$ is a finite-dimensional vector space that is closed under partial differentiation with respect to $x_1,\ldots,x_d$. Examples of sets $\mathscr{P}$ are polynomials on $x_1,\ldots,x_d$ with a total order less or equal to some $k\in\mathbb{N}$ and functions of the form $f(x) = \sum_{j=1}^{n_k} a_j \text{cos}(k_j^\top x) + b_j \text{sin}(\ell_j^\top x),$ with $a_j,b_j \in \mathbb{R}, k_j,\ell_j \in \mathbb{Z}^d,$ and $n_k\in \mathbb{N}$. We use the following assumptions on the model \eqref{eqs:nonlinear_SDE}.
\begin{asm}\label{asm:f_h_sigma_is_polynomial}
	Elements of functions $f, h,$ and $\varrho\varrho^\top$ belong to $\mathscr{P}$. 
\end{asm}
\begin{asm}\label{asm:c_monomials}
	The natural statistics $\left\{ c_i \right\}$  are selected as linearly independent elements of $\mathscr{P}$.
\end{asm}
\begin{asm}\label{asm:EM_c_ast}
	Each element of $h$ is in the span of $\left\{ 1, c_1, \ldots, c_m \right\}$. This means there exists $\lambda_k \in \mathbb{R}^{d_y}$ for $k=0,\ldots,m$, such that $h = \lambda_0 + \sum_{k=1}^{d_y}\lambda_k c_k$.
\end{asm}

Under Assumption \ref{asm:EM_c_ast}, the resulting exponential family of probability densities is known as the EM$(c^\ast)$ family \cite{brigo1999}. 
For this specific family, the projection filter equation reduces to \cite[Theorem 6.3]{brigo1999}
\begin{equation}
	d\theta_t = g(\theta_t)^{-1}\E_{\theta_t}\biggl[ \mathcal{L}\left[c\right]  - \frac{1}{2}(h^\top h) \left[ c - \eta(\theta_t) \right] \biggr] dt + \sum_{k=1}^{d_y} \lambda_k  dy_{t,k}, \label{eq:Brigo_dTheta_EM_c_ast}
\end{equation}

Assumptions \ref{asm:f_h_sigma_is_polynomial} and \ref{asm:c_monomials} ensure that every element of $\{ \mathcal{L}[c], h^\top h, h^\top hc\}$ belongs to a larger vector space $\tilde{\mathscr{P}}$ spanned by the basis of $\mathscr{P}$ and a finite number of multiplications of basis functions of $\mathscr{P}$. 
Explicitly, we can write $\mathcal{L}[c] - \frac{1}{2}(h^\top h)c = a_0 + A_0 \tilde{c}$ and $\frac{1}{2}(h^\top h) = b_0 + b_h^\top\tilde{c}$, where $a_0 \in \mathbb{R}^{m}$, $A_0 \in \mathbb{R}^{m\times(m+m_h)}$, $b_0 \in \mathbb{R}$,  $b_h \in \mathbb{R}^{(m+m_h)}$, and $\tilde{c}^\top = [c^\top, c_h^\top]$. Note that $c_h:\mathbb{R}^d\rightarrow \mathbb{R}^{m_h}$ is the vector of the remaining statistics of $x$ that are linearly independent of the elements of $c$. Therefore, under these assumptions, Equation \eqref{eq:Brigo_dTheta_EM_c_ast} can be expressed as
\begin{align}
	d\theta_t = g(\theta_t)^{-1}\left[ a_0 + b_0 \eta(\theta_t) + M(\theta_t) \tilde{\eta}(\theta_t)\right] dt + \lambda dy_t,
	\label{eq:Brigo_dTheta_EM_c_ast_polynomials}
\end{align}
where $\lambda = [\lambda_1,\ldots,\lambda_m]^\top$, $M(\theta_t) =  A_0 + \eta(\theta_t) b_h^\top$, and $\tilde{\eta}(\theta_t)\coloneqq \mathbb{E}_{\theta_t}[\tilde{c}]$.

In order to solve \eqref{eq:Brigo_dTheta_EM_c_ast_polynomials}, we need to compute the expectation $\tilde{\eta}$ and the Fisher metric $g$. In the automatic projection filter \cite{emzir2023}, the Fisher metric $g$ is obtained by automatic differentiation of the approximated cumulant-generating function via Theorem~\ref{thm:kass_exponential_family_expectation}. The cumulant-generating function can approximated via Gauss--Chebyshev quadrature in the univariate case or via sparse-grid integration methods in the multivariate case. For the expected values of the extended statistics $\tilde{\eta}$, the automatic projection filter uses the following lemma to calculate the remaining expectations in \eqref{eq:Brigo_dTheta_EM_c_ast_polynomials} \cite[Proposition 5.2]{emzir2023}.
\begin{lem}\label{lem:ExpectationByExtension}
    Let $s(x)\colon \mathbb{R}^d \rightarrow \mathbb{R}$ be a statistic, linearly independent of the natural statistics $c(x)$. If there exists an open neighborhood about the zero $\Theta_0$ such that for $\tilde{c} \coloneqq [c^\top(x) \; s(x)]^\top$, the quantity $\tilde{\psi}(\tilde{\theta}) \coloneqq \log(\int_\mathcal{X} \exp(\tilde{c}^\top \tilde{\theta}) dx)< \infty$ for any $\tilde{\theta} \in \Theta \times \Theta_0$, then for any $\theta \in \Theta$
    \begin{equation*}
        \mathbb{E}_\theta [s] = \left.\pdv{\tilde{\psi}(\tilde{\theta})}{\tilde{\theta}_{m+1}}\right|_{\tilde{\theta} = \theta_\ast}, \quad \theta_\ast = \mqty[\theta \\ 0].
    \end{equation*}
\end{lem}

\section{Adaptive Parametric Bijections}\label{sec:ParametricBijections}

In this section, we present the key contributions of this paper. Before doing so, let us first see the difficulties for adopting a static bijection, as was done in \cite{emzir2023}.
\subsection{Challenges of static bijection}\label{sec:challenges}
Although the automatic projection filter has been shown to perform well with multivariate dynamics \cite{emzir2023}, it unfortunately requires a large number of quadrature nodes. This is due to the need for the accurate computation of the exponential of the cumulant-generating function \eqref{eq:cumulant-generating}, which is done via numerical integration. To compute the cumulant-generating function for a parametric density $p_\theta$ a fixed smooth bijection $\phi\colon \mathcal{D} \rightarrow \mathcal{X}$ is used, where $\mathcal{D} \coloneqq (-1,1)^d$ is the canonical hypercube. Because during filtering, the high-density region of the filtering density moves within $\mathbb{R}^d$, fixing the bijection can lead to numerical instabilities. To illustrate this, let us denote by $\tilde{X} = \{\tilde{x}_i\}_{i=1}^N$ the set of quadrature nodes in the canonical hypercube and consider the univariate case with $p_\theta(x) = \frac{1}{\sqrt{2\pi}} \exp(-(x-\theta)^2/2)$ and $\phi=\tanh^{-1}(\tilde{x})$. In the domain $\mathcal{D}$, when a substantial part of the bijected nodes $\phi(\tilde{X})$ lies outside the high-density interval, the function $\phi'(\tilde{x}) \, p_\theta(\phi(\tilde{x}))$ moves to edge of $\mathcal{D}$, and is almost zero at the other edge. Fig.~\ref{fig:bijection_flaw_16} illustrates this effect with $\theta$ being either $0$ or $\pi/2$, and by using only 16 quadrature nodes, $\left\{ \phi(\tilde{x}_i) \right\}$ does not cover the region of high $p_\theta$ value as $\theta$ moves away from zero. Even worse, when $\theta$ is far from zero, increasing the number of quadrature nodes does not noticeably decrease the integration error. We show how to address this problem in the next section.

\begin{figure} 
    \centering
    \subfloat[$\theta=0$]{\includegraphics[width=0.25\textwidth]{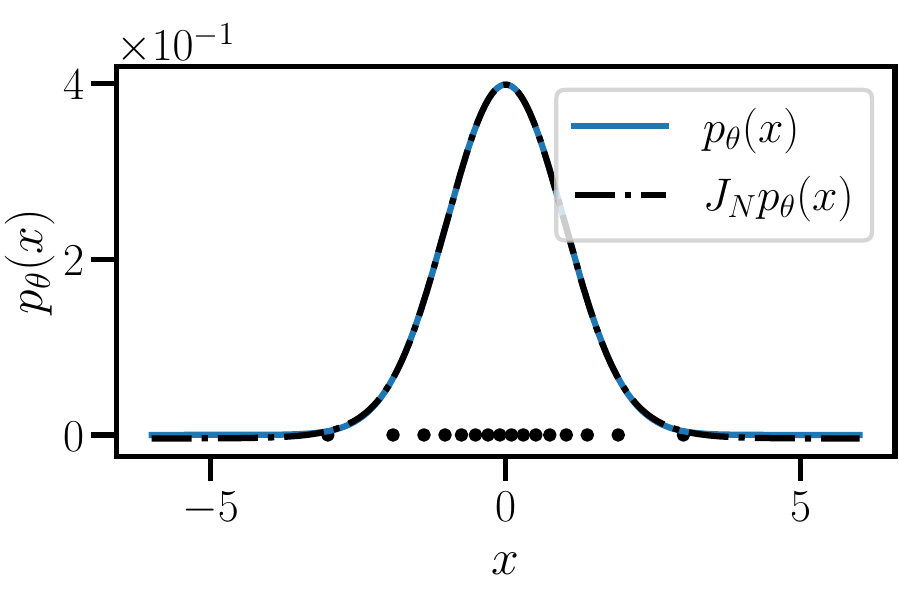}}
    \subfloat[$\theta=0$]{\includegraphics[width=0.25\textwidth]{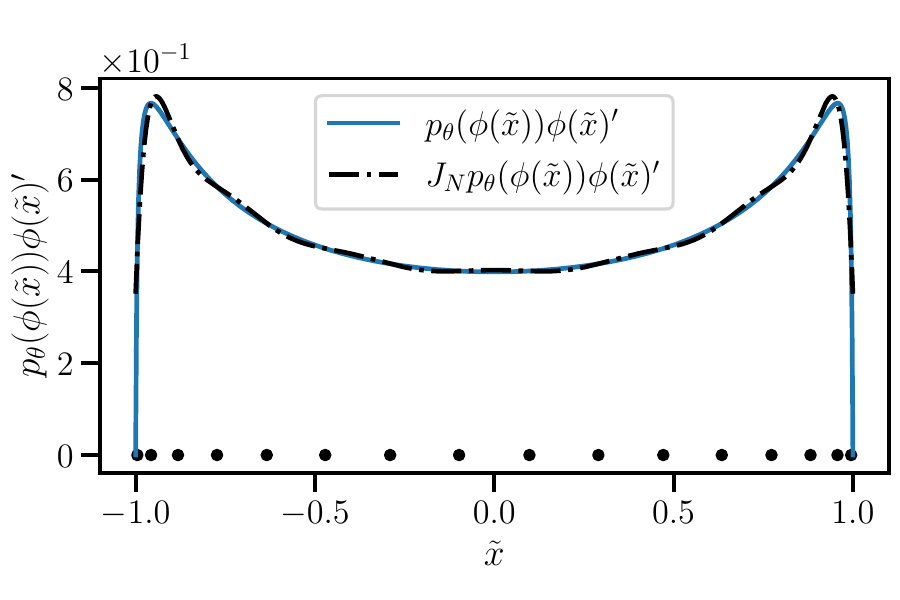}}\\
    \subfloat[$\theta=\pi/2$]{\includegraphics[clip,width=0.25\textwidth]{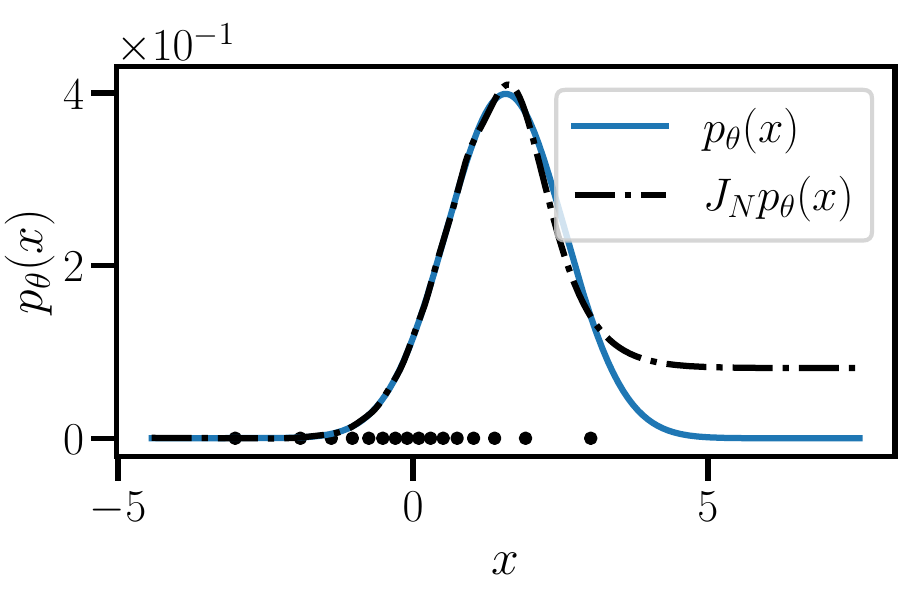}}
    \subfloat[$\theta=\pi/2$]{\includegraphics[clip,width=0.25\textwidth]{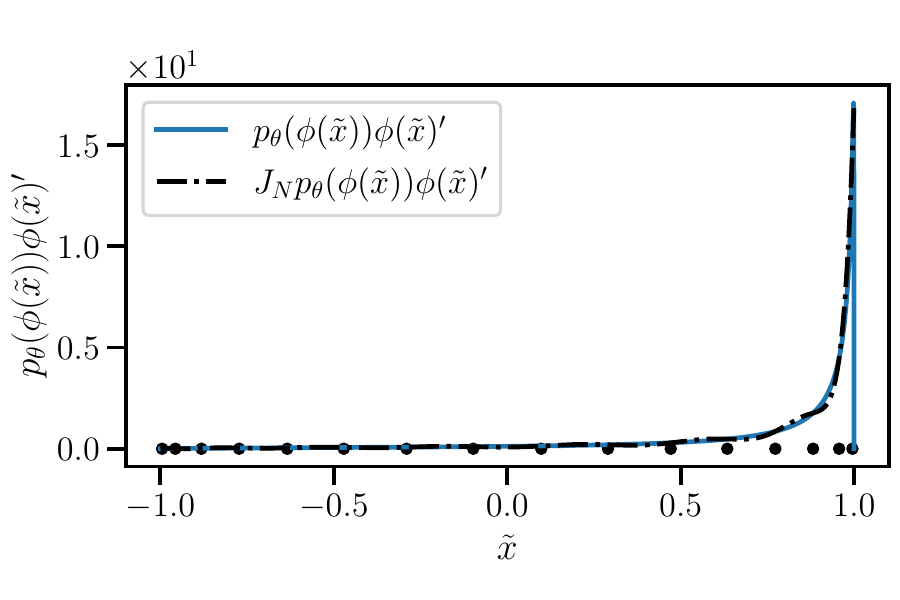}}

    \caption{Illustration of 16 bijected Gauss--Chebyshev's quadrature nodes, parametric density $p_\theta$, and interpolated approximation using Chebyshev's nodes with the bijection $\phi=\tanh^{-1}$. Figs. (a) and (c) are in the original sample space $\mathcal{X}$ and Figs. (b) and (d) are in the bijected sample space $\phi(\mathcal{X})$. When $\theta$ is shifted quite far from zero, the quadrature nodes do not entirely cover most of the high-density region of $p_\theta$.}
    \label{fig:bijection_flaw_16}
\end{figure}

\subsection{First Gaussian-Based Parametric Bijection}\label{sec:first_bijection}
The aforementioned challenges motivates us to propose a new parametric bijection $\phi_\xi$, where the bijection parameter $\xi\in \mathbb{R}^{n_\xi}, n_\xi \in \mathbb{N}$ can be used to reduce the numerical integration error. Let us first consider the univariate case and introduce a parametric probability density $q_\xi$ with support $\mathbb{R}$. Then
    $\zeta_\xi(x) \coloneqq 2 \int_{-\infty}^x q_\xi(y) dy - 1 $ %
is a valid bijection from $\mathbb{R} \rightarrow \mathcal{D}$. %

Now let us choose an exponential family manifold with $m$ natural parameters.
Choosing $\phi_\xi = \zeta_\xi^{-1}$, we can rewrite the exponential {of the cumulant-generating function} as
\begin{align}
    \exp(\psi(\theta)) &= 
     \int_\mathcal{D} \exp(c(\phi(\tilde{x}))^\top \theta) \dfrac{1}{2 q_\xi(\phi(\tilde{x}))} d\tilde{x}. \label{eq:integration_in_canonical_cube}
\end{align}

Instead of directly numerically computing this function using its definition \eqref{eq:cumulant-generating}, it is beneficial to use this transformed integral representation instead. The accuracy of the quadrature approximation to this integral depends on the choice of $q_\xi$. Recall that a Gaussian quadrature rule with $N$ quadrature nodes is exact when the integrand is chosen to be a polynomial of order $2N-1$ or less \cite{gautschi2004}. Therefore, in order to get an accurate integration result, we can choose a $q_\xi$ such that the integrand is close to a polynomial of order $2N-1$ or less. In fact, if we chose $q_\xi$ to be exactly $p_\theta$, then we would have the integrand equal to $\frac{1}{2} \exp(\psi(\theta))$. This would mean that the integration result would be exact, thus avoiding any integration errors. 
However, in the following proposition we show that there is no such a bijection that can be evaluated explicitly and makes the integrand in \eqref{eq:integration_in_canonical_cube} a polynomial.

\begin{prop}\label{prop:bijection_polynomial}
    Let $\phi$ be a bijection from $\mathcal{D}$ to $\mathbb{R}$ and let $\mathrm{EM}(c)$ be the exponential family given by \eqref{eq:EM_c} with parameters $\theta \in \Theta$, where $\Theta$ is given by \eqref{eq:Theta}. If there is no explicit antiderivative for $\beta(x)\coloneqq\exp\bigl(c(x)^\top\theta\bigr)$ as a function of $x$, then there is no explicit bijection $\phi$ such that $\beta(\phi(\tilde{x})) \, \phi(\tilde{x})'$ is a polynomial of $\tilde{x}$.
\end{prop}
\begin{proof}
    Suppose $\beta(\phi(\tilde{x})) \, \phi(\tilde{x})' = \rho_n(\tilde{x})$, where $\rho_n$ is a polynomial of order $n$ in $\tilde{x}$. Integrating both sides gives
    \begin{align*}
        \lim_{s_\ast \downarrow -1 }\int_{s=s_\ast}^{\tilde{x}}
        \beta(\phi(s)) \dv{\phi(s)}{s} \, ds &= \lim_{s_\ast \downarrow -1 }\int_{s=s_\ast}^{\tilde{x}} \rho_n(s) \, ds,\\
        \lim_{s_\ast \downarrow -1 }\int_{\phi=\phi(s_\ast)}^{\phi(\tilde{x})}
        \beta(\phi) \, d\phi &= \int_{s=-1}^{\tilde{x}} \rho_n(s) \, ds
        =\bar{\rho}_n(\tilde{x}),
    \end{align*}
    where $\bar{\rho}_n$ is a polynomial because $\rho_n$ is. Since the antiderivative of $\beta$ has no explicit form, there is no explicit form for $\phi$ neither.
\end{proof}

Even if the explicit antiderivative of $\beta(x)$ does exist, that is, $\int \beta(x) \, dx =\bar{\beta}(x) + C$, then solving for the bijection $\phi$ might still be impossible as $\bar{\beta}(\phi)$ might not be invertible. Therefore, instead of aiming at zero integration error, we focus on finding a parametric bijection, or a family of them, which leads to a smaller integration error. For this purpose, let us use a $d$-dimensional numerical quadrature with a positive weight function $\omega(\tilde{x})$ and with $N$ quadrature nodes as
\begin{equation}
    Q^{d,\omega}_N g \coloneqq \int_\mathcal{D} g(\tilde{x}) \, \omega(\tilde{x}) \, d\tilde{x} \approx \sum_{i=1}^N w_i \, g(\tilde{x}_i),
    \label{eq:Numerical_integration}
\end{equation}
where $\tilde{x}_i$ are the nodes and $w_i$ are the weights of the quadrature rule. The weighting function $\omega(\tilde{x})$ varies for different quadrature schemes. 
It is well-known that for a univariate Gauss-type quadrature, the quadrature nodes are distinct, and the weights $\{w_i\}_{i=1}^N$ are positive \cite[Theorem 1.46]{gautschi2004}. 
As a generalization of one-dimensional bijection $\phi_\xi$, where $\pdv{\phi_\xi}{\tilde{x}}=(2 q_\xi(x))^{-1}$, in multivariate context, the bijection operator (also denoted as $\phi_\xi$) is constructed from $q_\xi$, where $\abs{\det \pdv{\phi_\xi}{\tilde{x}}} = (2^d q_\xi(x))^{-1}$.
By using the quadrature rule, we can approximate the expectation of an arbitrary function $f(x)$ with respect to the density $p_\theta$ by
\begin{equation}
    \expvalb{\theta,N}{f;\xi} \coloneqq Q^{d,\omega}_N \left[ f(\phi_\xi) u(\phi_\xi) \omega^{-1}\right] \label{eq:Numerical_expectation} %
\end{equation}
where $u(x) \coloneqq p_\theta(x) / ( 2^d q_\xi(x) )$. 

The operator $f \mapsto \E_{\theta,N}[f;\xi]$ is linear, but it does not guarantee that $\E_{\theta,N}[1;\xi]=1$. The following lemma lists four important properties of this operator that we will use in the subsequent developments.

\begin{lem}\label{lem:numerical_expectation_identities}
    If the approximation of the cumulant-generating function using $N$ quadrature nodes is given by
        $\psi(\theta)^{(N)} = \log(Q^{d,\omega}_N\left[ \exp(c(\phi_\xi)^\top\theta) \abs{\det \pdv{\phi_\xi}{\tilde{x}}} \omega^{-1} \right])$, %
    then the following hold:
    \begin{enumerate}
        \item $\expvalb{\theta,N}{1;\xi} = \dfrac{\exp(\psi(\theta)^{(N)})}{\exp(\psi(\theta))}$,
        \item $\dfrac{\partial \psi(\theta)^{(N)}}{\partial \theta_j} = \dfrac{\expvalb{\theta,N}{c_j;\xi}}{\expvalb{\theta,N}{1;\xi}}$,
        \item $\dfrac{\partial^2 \psi(\theta)^{(N)}}{\partial \theta_j \partial \theta_k} = \dfrac{\expvalb{\theta,N}{c_j c_k;\xi}}{\expvalb{\theta,N}{1;\xi}} - \dfrac{\expvalb{\theta,N}{c_j;\xi}}{\expvalb{\theta,N}{1;\xi}} \dfrac{\expvalb{\theta,N}{c_k;\xi}}{\expvalb{\theta,N}{1;\xi}} $,
        \item $\dfrac{\expvalb{\theta,N}{f;\xi}}{\expvalb{\theta,N}{1;\xi}} = \sum_{i=1}^N w_i f(\phi_\xi(\tilde{x}_i))\\
        \times \exp(c(\phi(\tilde{x}_i)^\top \theta - \psi(\theta)^{(N)})) \abs{\det \pdv{\phi_\xi}{\tilde{x}}} \omega(\tilde{x}_i)^{-1}$.
    \end{enumerate}
\end{lem}
\begin{proof}
    For the first equality, using the definition \eqref{eq:Numerical_expectation}, we get:
    \begin{align*}
        \expvalb{\theta,N}{1;\xi} &= Q^{d,\omega}_N\left[ u(\phi_\xi) \omega^{-1} \right],\\
        &= Q^{d,\omega}_N \left[ \exp(c(\phi_\xi)^\top\theta - \psi(\theta)) \abs{\det \pdv{\phi_\xi}{\tilde{x}}} \omega^{-1} \right]\\
        &= \frac{\exp(\psi(\theta)^{(N)})}{\exp(\psi(\theta))}.
    \end{align*}
    The second and third equalities follow directly by taking derivatives of $\psi(\theta)^{(N)}$ with respect to $\theta_j$, and with respect to $\theta_j$ and $\theta_k$. The fourth equality follows from the first.
\end{proof}

A natural way to choose the parameter $\xi$ is by minimizing the square of numerical integration error of \eqref{eq:integration_in_canonical_cube}, that is, $\bigl(\exp(\psi(\theta))-\exp \bigl(\psi(\theta)^{(N)}\bigr)\bigr)^2$. Let us define the integration error $E_N[f(x);\xi]$ as follows:
\begin{align}
    E_N[f;\xi] &\coloneqq \E_\theta[f] - \E_{\theta,N}[f; \xi]. \label{eq:En_f}
\end{align} 
Then, since $\exp(\psi(\theta))-\exp\bigl(\psi(\theta)^{(N)}\bigr) = \exp(\psi(\theta))(1-\expvalb{\theta,N}{1;\xi})$, minimizing the square of the numerical integration error of \eqref{eq:integration_in_canonical_cube} is equivalent to minimizing $E_N[1;\xi]^2$. The following proposition gives a necessary and sufficient condition such that the squared error $E_N[1;\xi]^2$ is a convex function. We denote the partial ordering of two squared matrices as $A\succeq B$ if $A-B$ is a positive semidefinite matrix. 
\begin{prop}\label{prop:convexity_of_abs_En}
    Let $u(x) = \frac{p_\theta(x)}{2^d q_\xi(x)}$, where $q_\xi = \exp\bigl(\tilde{c}(x)^\top\xi - \psi(\xi)\bigr)$, $c(x)$ and $\tilde{c}(x)\in \mathcal{C}^2(\mathbb{R})$ are the natural statistics of $p_\theta$ and $q_\xi$, respectively, $\psi(\xi)$ is the cumulant-generating function corresponding to $q_\xi$, and $\phi_\xi$ is a smooth bijection from $\mathcal{D}$ to $\mathbb{R}^d$ constructed from $q_\xi$. Also, let $\Xi \coloneqq \left\{ \xi \in \mathbb{R}^{n_\xi}\colon \psi(\xi) <\infty \right\}$. Then $E_N[1;\xi]^2$ is convex on an open convex set $\Xi_0 \subseteq \Xi$, if and only if for any $\xi \in \Xi_0$, the following condition is satisfied:
    \begin{equation}
        \expvalb{\theta,N}{\frac{1}{u}\dv{u}{\xi};\xi}\expvalb{\theta,N}{\frac{1}{u}\dv{u}{\xi};\xi}^\top
        \succeq E_N[1;\xi] \expvalb{\theta,N}{\frac{1}{u}\dv[2]{u}{\xi};\xi}.\label{eq:E_n_square_convex}
    \end{equation}
\end{prop}
\begin{proof}
    Using the definitions of $E_N[1;\xi]$ and $\expvalb{\theta,N}{\cdot}$ from \eqref{eq:En_f} and \eqref{eq:Numerical_expectation} respectively, we can write
    \begin{align}
        \frac{1}{2}\pdv{E_N[1;\xi]^2}{\xi} = -E_N[1;\xi] \expvalb{\theta,N}{\frac{1}{u}\dv{u}{\xi};\xi} \label{eq:gradient_E_N_1_squared}
    \end{align}
    and
    \begin{align}
        \frac{1}{2}\pdv[2]{E_N[1;\xi]^2}{\xi} &= \expvalb{\theta,N}{\frac{1}{u}\dv{u}{\xi};\xi}\expvalb{\theta,N}{\frac{1}{u}\dv{u}{\xi};\xi}^\top \nonumber \\
        &\quad- E_N[1;\xi] \expvalb{\theta,N}{\frac{1}{u}\dv[2]{u}{\xi};\xi}.\label{eq:Hessian_E_N_1_squared}
    \end{align}
    Therefore, if for any $\xi \in \Xi_0$, condition \eqref{eq:E_n_square_convex} is satisfied, then $\pdv[2]{E_N[1;\xi]^2}{\xi}$ is positive semidefinite. By \cite[Theorem 4.5]{rockafellar1970} $E_N[1;\xi]^2$ is a convex function on $\Xi_0$. The necessary part can be obtained using a similar argument to the proof of \cite[Theorem 4.5]{rockafellar1970}. 
\end{proof}
We can now highlight two difficulties of using gradient-based method to find the minimizer of $E_N[1,\xi]^2$. First, by parts 1) and 4) of Lemma \ref{lem:numerical_expectation_identities}, the Jacobian and Hessian of $E_N[1;\xi]^2$ with respect to $\xi$, which are given respectively by \eqref{eq:gradient_E_N_1_squared} and \eqref{eq:Hessian_E_N_1_squared}, cannot be explicitly calculated unless the cumulant-generating function $\psi(\theta)$ is known in a closed form. Secondly, it is impossible to ensure the criterion that $E_N[1,\xi]^2$ is locally convex even on some bounded interval since its Hessian cannot be evaluated. As a consequence, using gradient methods or (quasi-)Newton methods for finding a locally optimal $\xi$ is not feasible.

With that being said, a more promising approach is to select $\xi$ such that $q_\xi$ covers the high-density area of $p_\theta$ as tightly as possible by optimizing another criterion as explained below, and the resulting bijection $\phi_\xi$ should then be amenable to direct computation. This turns out to be a viable approach to take, and the thus-constructed bijection could be shown to work in practice as well as be optimal with respect to the squared integration error, under some technical conditions. 
Our starting point is the following lemma \cite{herbrich2005}.
\begin{lem}\label{lem:Herbrich_KL_divergence}
    Let $q_\xi$ be a density from {an} exponential family with natural statistics given by $\tilde{c}(x)$. For any distribution $p$, the distribution $q_\xi^\ast$ that minimizes $\KL(p||q_\xi)$, satisfies
    \begin{align}
        \E_{q_\xi^\ast}\left[ \tilde{c} \right] = \E_{p}\left[ \tilde{c} \right]. \label{eq:moment_matching}
    \end{align}
\end{lem}
Essentially, the exponential density $q_\xi$ that minimizes the $\KL(p||q_\xi)$ distance satisfies the moment-matching equality \eqref{eq:moment_matching}. 
If we choose $q_\xi$ to be a Gaussian density, then the mean $\mu$ and variance $\Sigma$ of $q_\xi$ should satisfy
\begin{subequations}
    \begin{align}
        \mu_i &= \E_{\theta}\left[x_i\right], & i=1,\ldots,d,\\
        \Sigma_{ij} &= \E_{\theta}\left[(x_i-\mu_i)(x_j-\mu_j)\right] , &i,j=1,\ldots,d.
    \end{align}\label{eq:mu_sigma}
\end{subequations}
In the actual implementation, we replace {the cumulant-generating function by its approximation $\psi(\theta)^{(N)}$ that is obtained from numerical integration using $N$ quadrature nodes}. Since the approximation errors of the parameters $\mu$ and $\Sigma$, as well as the cumulant-generating function $\psi(\theta)$, depend linearly on the integration errors on $\int \exp\bigl(c^\top x\bigr) dx$, then selecting an $N$ such that all the approximation errors for $\mu$, $\Sigma$, and $\psi(\theta)$ are simultaneously kept below their acceptable upper limits, is feasible (see \cite[Theorem 3.1]{emzir2023}).

In the following proposition, we show how to construct a smooth bijection $\zeta_\xi$ from $\mathbb{R}^d$ to $\mathcal{D}$ using a parametric density $q_\xi$. We also show that its inverse $\phi_\xi =\zeta_\xi^{-1}$ can be expressed in a closed form when using Gaussian density $q_\xi$ with parameters in \eqref{eq:mu_sigma}.
\begin{prop}\label{prop:zeta_xi_gaussian}
    Let $s\colon\mathbb{R}^d \times \mathcal{P} \to \mathcal{D}$ defined by 
    \begin{subequations}
    \begin{align}
        s(z,q_\xi) &= \mqty[s_1(z_1,q_\xi)\\ \vdots \\ s_d(z_d,q_\xi)], \label{eq:s_general}\\
        s_i(z_i,q_\xi) &= 2 \dfrac{\int_{-\infty}^{z_i} q_\xi(z_1,\ldots,y_i,\ldots,z_{d}) dy_i}{\int_{-\infty}^{\infty} q_\xi(z_1,\ldots,y_i,\ldots,z_{d}) dy_i} - 1.
    \end{align}       
    \end{subequations}
    If $q_\xi$ is a smooth density with support equal to $\mathbb{R}^d$, then $s(\cdot, q_\xi)$ is a smooth bijection from $\mathbb{R}^d$ onto $\mathcal{D}$.
    
    Assume that $q_\xi$ is a Gaussian density with mean $\mu$ and covariance matrix $\Sigma \succ 0$, and let the eigendecomposition of $\Sigma$ be given by $\Sigma = T^{-1} \Lambda T$, where $T$ is unitary. Let $\tilde{q}_\xi$ be another Gaussian density with mean at $T\mu$ and variance $\Lambda$. Define $\zeta_\xi: \mathbb{R}^d \to \mathcal{D}$ as
        $\zeta_\xi(x) = s(Tx, \tilde{q}_\xi)$. %
    Then $\abs{\det\pdv{\zeta_\xi}{x}} = 2^d q_\xi$. Moreover, the inverse of $\zeta_\xi$, denoted as $\phi_\xi:\mathcal{D}\to\mathbb{R}^d$, is given by
    \begin{align}
        \phi_\xi(\tilde{x}) = \mu + \sqrt{2} T^{-1} \Lambda^{1/2} \erf^{-1}(\tilde{x}), \quad \tilde{x}\in \mathcal{D} \label{eq:phi_xi_gaussian}
    \end{align}
    where $\erf^{-1}(\tilde{x}) = \mqty[\erf^{-1}(\tilde{x}_1),\ldots,\erf^{-1}(\tilde{x}_d)]^\top$.
\end{prop}
\begin{proof}
    Each $s_i(z_i,q_\xi)$ is a smooth bijection from $\mathbb{R}$ to $(-1,1)$. By the definition of $s$, for each $\tilde{x}\in \mathcal{D}$ there exists $z \in \mathbb{R}^d$ such that $s(z,q_\xi)=\tilde{x}$. Let $z,y \in \mathbb{R}^d$, and suppose $s(z,q_\xi) = s(y,q_\xi)$. Since $s_i(z_i,q_\xi)$ is a smooth bijection, $z_i$ must equal to $y_i$, which leads to $z=y$. The smoothness of $\zeta_{\xi}$ follows from that of $s_i(z_i,q_\xi)$ for each $i$.\looseness=-1

    For the second part, it can be verified that with $z=Tx$ and $\mu_z = T\mu$ we have
        $
        s_i(z_i,\tilde{q}_\xi)  = \erf (\frac{\Lambda_{ii}^{-1/2}(z_i-\mu_{z,i})}{\sqrt{2}}). %
        $
    Then since $T$ is unitary, we get $\abs{\det\pdv{\zeta_\xi}{x}} = \abs{\det \pdv{s(z,\tilde{q}_\xi)}{z}} = \prod_{i=1}^d 2 \frac{1}{\sqrt{2 \pi \Lambda_{ii}}} \exp(-\frac{1}{2}(z_i-\mu_{z,i})^2) = 2^d q_\xi$.
    The inverse of $\zeta_\xi$ can be obtained directly from the definition of $s_i$%
\end{proof}

Proposition \ref{prop:zeta_xi_gaussian} tells us that for any non-degenerate Gaussian density $q_\xi$, there exists a bijection $\zeta_\xi$ such that the transformation of an infinitesimal volume $dx$ on $\mathbb{R}^d$ under $\zeta_\xi$ is equivalent to $2^d q_\xi$ times $d\tilde{x}$. The bijection $\phi_\xi$ given by \eqref{eq:phi_xi_gaussian} can be interpreted as the following consecutive operations. First it transforms the quadrature nodes from the canonical hypercube to $\mathbb{R}^d$ by the inverse error function, then it scales each axis by the appropriate square roots of the eigenvalues of $\Sigma$. Afterwards, it rotates the quadrature nodes according to the unitary matrix $T$. Lastly it shifts the quadrature nodes by $\mu$ from the origin; see the middle columns of Figures \ref{fig:particle_vs_projection_densities} for an illustration of these operations. In the numerical implementation, we will use this bijection to project the Gauss--Chebyshev nodes (for univariate case), or the sparse Gauss--Patterson nodes (for multivariate case), from $\mathcal{D}$ to $\mathbb{R}^d$ (see \cite{gerstner1998} for details). We will refer to these numerical integrations as the Gauss--Chebyshev quadrature (GCQ) and the Gauss--Patterson quadrature (GPQ), respectively.

Let us turn our attention to checking if the parameters $\xi$ that are selected via the moment-matching rule \eqref{eq:moment_matching} also optimize $E_N[1;\xi]^2$ locally. Proposition \ref{prop:optimality} below shows that if we use an approximated version of the moment-matching criterion \eqref{eq:moment_matching} for a general exponential family $\text{EM}(\tilde{c})$ (it does not have to be a Gaussian family), then the selected parameters $\xi$ also optimize $E_N[1;\xi]^2$ under certain constraints in the numerical expectation. Explicitly, in a special case where both natural statistics $c$ and $\tilde{c}$ are equivalent, this constraint requires that the numerical expectation $\expvalb{N,\theta}{c;\xi}/\expvalb{N,\theta}{1;\xi}$ is equal to the true value of $\expvalb{\theta}{c}$.

\begin{prop}\label{prop:optimality}
    Let $q_\xi$ in \eqref{eq:s_general} be a density from $\text{EM}(\tilde{c})$ where the natural statistics $\tilde{c}_i \in \mathcal{C}^2(\mathbb{R}^d)$ are linearly independent. The parameter $\xi$ optimizes $E_N[1;\xi]^2$ if
    $\pdv{\psi(\xi)}{\xi} -\frac{\expvalb{\theta,N}{\tilde{c};\xi}}{\expvalb{\theta,N}{1;\xi}} + \frac{\expvalb{\theta,N}{\dv{x}{\xi} \dv{m}{x};\xi} }{\expvalb{\theta,N}{1;\xi}} = 0,$ where 
    $m(\xi) \coloneqq c^\top(x(\xi))\theta-\tilde{c}^\top(x(\xi)) \xi -(\psi(\theta)-\psi(\xi))$.
    In particular, if the following approximated moment-matching rule is used to choose the parameter $\xi$:
    \begin{align}
        \dfrac{\expvalb{\theta,N}{\tilde{c};\xi}}{\expvalb{\theta,N}{1;\xi}} =& \pdv{\psi(\xi)}{\xi}, \label{eq:approximated_moment_matching}
    \end{align}
    then, the selected parameter $\xi$ is a local optimum of $E_N[1;\xi]^2$ if 
    $    \expvalb{\theta,N}{\dv{x}{\xi} \dv{m}{x} ;\xi} = 0. $ %
    If $\tilde{c}=c$ and $\xi = \theta$, then $\xi$ is a local optimum of  $E_N[1;\xi]^2$.
\end{prop}
\begin{proof}
    Let us denote $x(\xi)\coloneqq \phi_\xi(\tilde{x})$. Using \eqref{eq:gradient_E_N_1_squared} 
    from the proof of Proposition \ref{prop:convexity_of_abs_En},  
    we can write
    \begin{align*}
       \frac{1}{2 \expvalb{\theta,N}{1;\xi}} \dv{E_N[1;\xi]^2}{\xi} =& -[\pdv{\psi(\xi)}{\xi} -\frac{\expvalb{\theta,N}{\tilde{c};\xi}}{\expvalb{\theta,N}{1;\xi}} \\
       &+ \frac{\expvalb{\theta,N}{\dv{x}{\xi} \dv{m}{x};\xi} }{\expvalb{\theta,N}{1;\xi}}]E_N[1;\xi].
    \end{align*}
    Hence, if the approximated moment-matching rule \eqref{eq:approximated_moment_matching} is satisfied, and $\expvalb{\theta,N}{\dv{m}{x}\dv{x}{\xi};\xi} = 0$, then $\frac{1}{2} \dv{E_N[1;\xi]^2}{\xi} = 0$, which ensures the local optimality. The case when $\tilde{c}=c$ follows directly by substitution.
    
\end{proof}
Note that the approximated moment-matching rule \eqref{eq:approximated_moment_matching} can be implemented as an iterative procedure, where using initial parameters $\xi_i$, one computes the updated parameters $\xi_{i+1}$ via $\frac{\expvalb{\theta,N}{\tilde{c};\xi_i}}{\expvalb{\theta,N}{1;\xi_i}} = \pdv{\psi(\xi_{i+1})}{\xi_{i+1}}$. The update is repeated until the current iterate is close enough to the fixed point. To analyze the convergence of this iterative procedure, let us denote $\hat{\xi} \coloneqq \pdv{\psi(\xi)}{\xi}$. The mapping from $\xi$ to $\hat{\xi}$ is a diffeomorphism via the Legendre transformation \cite[Theorem 2.2.3]{kass1997}. %
Therefore, we can write $F_N(\hat{\xi}) \coloneqq \frac{\expvalb{\theta,N}{\tilde{c};\xi}}{\expvalb{\theta,N}{1;\xi}}$. The approximated moment-matching rule can then be written as a Picard iteration $\hat{\xi}_{i+1} = F_N(\hat{\xi}_i)$. Using Banach’s fixed-point theorem, in the proposition below, we show that there exists a fixed point of the mapping $F_N$ on some subset of $\hat{\Xi}\coloneqq\{ \hat{\xi}\colon\xi \in \Xi \}$.
\begin{prop}\label{prop:Fixed_point}
    Using the notations of Proposition \ref{prop:optimality}, suppose that for each $\theta \in \Theta$, $\tilde{c}_i(\phi_\xi)u(\phi_\xi)\omega^{-1}$ belongs to $\mathcal{W}^r_d(\mathcal{D})$ for any $i\in \left\{1,\ldots,n_\xi  \right\}$ and for some $r>0\in \mathbb{N}$, uniformly for any $\xi \in \Xi_\theta$, where,
        $\mathcal{W}^r_d(\mathcal{D}) \coloneqq \{ f\colon \mathcal{D} \to \mathbb{R}\colon \|\frac{\partial^{\abs{s}} f}{\partial\tilde{x}_{(1)}^{s_1}\ldots \partial\tilde{x}_{(d)}^{s_d}}\|_\infty < \infty, s_i < r\}$,
    with $\abs{s}=\sum_{i=1}^d s_i$ and $\Xi_\theta  \subset \Xi$ open. Moreover, assume that $\tilde{c}_i(\phi_\xi)u(\phi_\xi)$ is continuously differentiable in $\xi$ on $\Xi_\theta $. Then there exists an $N_0 \in \mathbb{N}$ and a subset $\hat{\Xi}_c \subset \hat{\Xi}$ such that $F_N$ is a contraction in $\hat{\Xi}_c$ for $N \ge N_0$.
\end{prop}
\begin{proof}
    $F_N$ is continuously differentiable on $\Xi_\theta$ by the assumption $\tilde{c}_i(\phi_\xi)u(\phi_\xi)$ is continuously differentiable in $\xi$ on $\Xi_\theta $ and the definition of $F_N$. Using the mean value theorem \cite[\textsection 0.27]{aubin2001} we can write $\|F_N(\hat{\xi}_1)-F_N(\hat{\xi}_2)\| \leq \sup_{\hat{\xi}\in\hat{\Xi}_c}\|\pdv{F_N}{\hat{\xi}}\| \|\hat{\xi}_1-\hat{\xi}_2\|$. Therefore, if $\sup_{\hat{\xi}\in\hat{\Xi}_c}\|\pdv{F_N}{\hat{\xi}}\| < 1$ in an open convex subset $\hat{\Xi}_c$ then $F_N$ is a contraction in the set. Notice that using $\pdv{\xi}{\hat{\xi}} = (\pdv[2]{\psi(\xi)}{\xi})^{-1}$ \cite[p. 17]{amari2016}, we obtain
    $\pdv{F_N}{\hat{\xi}} = ( \pdv[2]{\psi(\xi)}{\xi})^{-1} \pdv{F_N}{\xi}$. Since $\{ \tilde{c}_i \}$ are linearly independent, then $( \pdv[2]{\psi(\xi)}{\xi})$ is invertible for any $\xi \in \Xi$ \cite{calin2014}. Therefore, finding $\hat{\Xi}_c$ such that $\sup_{\hat{\xi}\in\hat{\Xi}_c}\|\pdv{F_N}{\hat{\xi}}\| < 1$ is equivalent to finding an open convex subset $\Xi_c \subset \Xi_\theta$ such that 
    \begin{align}
        \sup_{\xi \in \Xi_c}\norm{\left( \pdv[2]{\psi(\xi)}{\xi} \right)^{-1} \pdv{F_N}{\xi}}<1. \label{eq:Xi_c}    
    \end{align}
    Since
        $\sup_{\xi \in \Xi_c}\|{( \pdv[2]{\psi(\xi)}{\xi})^{-1} \pdv{F_N}{\xi}}\|$ is less than $\sup_{\xi \in \Xi_c}\|( \pdv[2]{\psi(\xi)}{\xi})^{-1}\|\sup_{\xi \in \Xi_c}\|\pdv{F_N}{\xi}\|,$
    and, for an invertible matrix $T$ we have $\frac{1}{\|T\|}<\|T^{-1}\|$, condition in  \eqref{eq:Xi_c} is satisfied if
        $\sup_{\xi \in \Xi_c}\|\pdv{F_N}{\xi}\| < \inf_{\xi \in \Xi_c}\|( \pdv[2]{\psi(\xi)}{\xi})^{-1}\|^{-1} < \inf_{\xi \in \Xi_c}\|(\pdv[2]{\psi(\xi)}{\xi})\|$.
    Since $\Xi_\theta$ is open, we can select an open subset $\Xi_a \subset \Xi_\theta$ away from the boundary of $\Xi$ such that there exists a positive $\alpha$ satisfying $\alpha I \prec \pdv[2]{\psi(\xi)}{\xi}  $ for any $\xi \in \Xi_a$. This is always possible since $\text{EM}(\tilde{c})$ is a regular exponential family, which means $\Xi$ is an open convex subset of $\mathbb{R}^{n_\xi}$ \cite{brown1986}. Since for any $i, \tilde{c}_i(\phi_\xi)u(\phi_\xi)\omega^{-1} \in \mathcal{W}^r_d(\mathcal{D})$ on $\Xi_\theta$, then $\|F_N(\hat{\xi}(\xi)) - \expvalb{\theta}{\tilde{c}}\| = \mathcal{O}(N^{-r}\log(N)^{(d-1)(r-1)})$ on $\Xi_\theta$ \cite[\textsection 4.1.1]{holtz2011}. As $N$ approaches infinity, the Jacobian $\pdv{F_N}{\xi}$ at any $\xi_a \in \Xi_a$ decreases to zero. Therefore, there exists $N_0$ and $r_1>0$ such that for any $\xi \in B(\xi_a,r_1) \subset \Xi_a$, the requirement $\|\pdv*{F_N}{\xi}\|<\alpha$ is satisfied for $N \ge N_0$ in $\hat{\Xi}_c \coloneqq \{ \hat{\xi}\colon \xi \in  B(\xi_a,r_1) \}$.
\end{proof}

The conditions in Proposition \ref{prop:Fixed_point} can be shown to be satisfied for some $r>0$ when $|p_\theta/(q_\xi^r)|$ is bounded and goes to zero as $\norm{x}\to\infty$ with $\Xi_\theta = \Xi$. This condition is valid for the three numerical simulations considered in Section \ref{sec:NumericalExamples}.

\subsection{Second Gaussian-Based Parametric Bijection}
In Section \ref{sec:first_bijection}, we have used the approximated moment-matching rule \eqref{eq:approximated_moment_matching} to construct the bijection from the hypercube $\mathcal{D}$ to $\mathbb{R}^d$. However, there are other quadrature methods that do not operate on the hypercube, for example, the Gauss--Hermite quadrature (GHQ) which operates on $\mathbb{R}^d$. To compute the cumulant-generating function efficiently using these numerical integration methods, we introduce another bijection where we specifically focus on the GHQ case. 

We use $N(d,l)$ to denote the number of quadrature nodes for a dimension $d$ and a level $l$, such that $N\coloneqq N(d,l)$. We obtain the sparse multivariate Gauss--Hermite nodes and weights $\left\{ \tilde{x}i,w_i \right\}_{i=1}^{N(d,l)}$ by using the Smolyak construction on one-dimensional Gauss--Hermite nodes (see \cite{bungartz2004} for details). We use a similar choice of univariate quadrature nodes per level as in GPQ, where, for a level $l\in \mathbb{N}$, the corresponding univariate nodes are given by $N(1,l) = 2^{l+1}-1$. This coincides with the choice of nodes introduced in \cite{jia2011}, however in their work, the quadrature levels greater than three are abandoned. Interestingly, if we select $N(1,l) = \mathcal{O}(2^l)$, then both GPQ and GHQ satisfy \cite{gerstner1998} $N(d,l) = \mathcal{O}(2^l l^{d-1})$.

The exponential of the cumulant-generating function is approximated as follows:

\begin{equation}
    \begin{split}
        &\int_{\mathbb{R}^d} \exp(c(x)^\top \theta) dx %
        \\
        &= \int_{\mathbb{R}^d} \exp(c(\phi_\xi(\tilde{x}))^\top\theta) \exp(\tilde{x}^\top \tilde{x}) \\
        &\qquad\times \abs{\det\pdv{\phi_\xi(\tilde{x})}{\tilde{x}}} \exp(-\tilde{x}^\top \tilde{x})  d\tilde{x}
        \approx \sum_{i=1}^{N(d,l)} w_i z(\tilde{x}_i), \label{eq:SparseGaussHermite}
    \end{split}
\end{equation}  
where, 
\begin{subequations}
    \begin{align}
        \phi_\xi(\tilde{x}) &\coloneqq \mu + \sqrt{2} T^{-1}\Lambda^{1/2} \tilde{x},\label{eq:Gauss_Hermite_Bijection}\\
        z(\tilde{x}) &\coloneqq \exp(c(\phi_\xi(\tilde{x}))^\top\theta) \exp(\tilde{x}^\top \tilde{x}) \abs{\det\pdv{\phi_\xi(\tilde{x})}{\tilde{x}}}.
    \end{align} 
\end{subequations}

In \eqref{eq:SparseGaussHermite}, $\left\{ \tilde{x}_i \right\}$ is the set of the Gauss--Hermite quadrature nodes, and $\left\{ w_i \right\}$ are their corresponding weights. We select $\mu$ and $\Sigma$ according to the moment-matching rule \eqref{eq:moment_matching}. Therefore, the bijection \eqref{eq:Gauss_Hermite_Bijection} also ensures that the sparse Gauss--Hermite quadrature nodes are always placed in the high-density domain in $\mathbb{R}^d$.

The GHQ scheme suffers from a very weak nesting capability since the intersection between roots of Hermite polynomials of successive orders contains only the origin $x=0$. This is in contrast with the GPQ scheme used in the previous section, which is highly efficient since it has a polynomial exactness up to order $3N+1$, in addition to being fully nested. Therefore, more integration nodes would be required by GHQ to achieve the same accuracy as GPQ (see \cite[Section 3.1.2]{gautschi2004}). A problem with high-order Gauss--Hermite methods is that the quadrature weights can easily lie below machine precision \cite{trefethen2022}. In some applications, fortunately, one can ignore the Gauss--Hermite nodes that have weights below machine precision and still obtain satisfying integration results. As we will see in Section \ref{sec:NumericalExamples}, with a similar sparse integration level to that of GPQ, the GHQ scheme combined with the adaptive bijection \eqref{eq:Gauss_Hermite_Bijection} might offer a competitive advantage compared to GPQ combined with \eqref{eq:phi_xi_gaussian} as their quadrature nodes spread wider than those of latter; see Figure \ref{fig:particle_vs_projection_densities}.

\subsection{Integrating the parametric bijection in the automatic projection filter}
We present Algorithm \ref{alg:projectionfiler} that combines the parametric bijection with the automatic projection filter algorithm. The bijections \eqref{eq:phi_xi_gaussian} and \eqref{eq:Gauss_Hermite_Bijection} have identical parameters $\mu$ and $\Sigma$, and thus can be implemented similarly.
\begin{algorithm}
\begin{algorithmic}[1]
	\Procedure{AutomaticProjectionFilter}{$p_\theta,\theta, \xi, t, dy_t$}
	\State $g\gets$  \Call{FisherMetric}{$p_\theta; \xi$} \Comment{Calculate Fisher Metric of $p_\theta$}
	\State $\tilde{\eta}\gets$  \Call{ExtendedExpectation}{$p_\theta; \xi$} \Comment{Calculate expectation of the extended natural statistics $\tilde{c}$}
	\State $\eta \gets \tilde{\eta}[:m]$  \Comment{The expectation of natural statistics $c$ is the first $m$ of $\tilde{\eta}$}
	\State $d\theta \gets$ \Call{FilterEquation}{$g,\xi,\tilde{\eta},\eta,t,dy_t$} \Comment Eq. \eqref{eq:Brigo_dTheta_EM_c_ast_polynomials}
	\State $\theta \gets \theta + d\theta$ \Comment{Update the parameter}
	\State $\xi \gets$ \Call{UpdateBijectionParameters}{$\eta, g $} \Comment{Eq. \eqref{eq:approximated_moment_matching}}
	\State \textbf{return} $\theta,\xi$\Comment{Return the updated  value of the natural parameter $\theta$ and bijection parameter $\xi$}
	\EndProcedure
\end{algorithmic}
\caption{Single step from the automatic projection filter using parametric bijection.\label{alg:projectionfiler}}
\end{algorithm}

\section{Numerical Examples}\label{sec:NumericalExamples}

In this section, we consider three numerical experiments to assess the effectiveness of our proposed bijections. For these examples, we choose $\mathscr{P}$ as the set of polynomials in $x_t$ with order less than or equal to some $n_p\in \mathbb{N}$. 

\subsection{Univariate example}
\newcommand{\onedsimulationid}{CubicSensor08_11_2022_16_17_33}

The first example is a scalar dynamical system with a nonlinear measurement model \cite{hanzon1991,brigo1999}:
\begin{subequations}
    \begin{align}
        dx_t &=  \kappa dt + \sigma dW_t,\label{eq:dx_one_d}\\
        dy_t &=  \beta x^3_t dt +  dV_t,
        \end{align}\label{eqs:one_d_simulation}
\end{subequations}
with two independent standard Wiener processes $\{ W_t , t \geq 0\}$ and $\{ V_t , t \geq 0\}$, and three constants given by the process noise constant $\sigma=0.4$, the drift constant $\kappa=0.25$, and the nonlinear measurement scale $\beta = 0.8$. It is well known that the optimal filter for this problem is infinite dimensional \cite{hazewinkel1983}. We use an exponential manifold with $c_i \in \{x, x^2, x^3, x^4\}$. We remark that the drift term in \eqref{eq:dx_one_d} makes the filtering problem significantly harder than that of \cite{emzir2023}. We choose the initial parameters of the projection filter as $\theta_0 = [0,2,0,-1]^\top$. This initial condition vector reflects exactly the initial density of the dynamical system which corresponds to a bimodal non-Gaussian density with peaks at $-1$ and $1$. We solve the Kushner–Stratonovich stochastic PDE on a uniform grid. The simulation time step is set to be $10^{-4}$. We generate one measurement realization with $x_0 = 1$. We compare three different bijections: The first one is a static bijection $\tanh(\tilde{x})^{-1}$, the second one is the Gaussian-based bijection \eqref{eq:phi_xi_gaussian}, and the third one is the Gauss--Hermite bijection \eqref{eq:Gauss_Hermite_Bijection}. Furthermore, we use $\Chebyshevbasislow$ and $\Chebyshevbasishigh$ integration nodes for the static bijection and $\Chebyshevbasislow$ such nodes for both parametrized bijections.

The simulation results, shown in Figure \ref{figs:comparison_one_d}, depict the evolution of the densities obtained with the finite difference approach and with the projection filters, along with their corresponding Hellinger distances, \textcolor{black}{where the Hellinger distance between two densities $p$ and $q$ is given by $\frac{1}{2}\int_\mathcal{X}(\sqrt{p}-\sqrt{q})^2 dx$}. Even though both projection filters based on parametric bijections \eqref{eq:phi_xi_gaussian} and \eqref{eq:Gauss_Hermite_Bijection} require less integration nodes than that based on the static bijection, they do however substantially reduce Hellinger distances compared to the finite difference approximation, as shown in Figure \ref{figs:comparison_one_d}. The projection filter based on the bijection \eqref{eq:phi_xi_gaussian} with GCQ produces better approximated densities than the GHQ scheme across the entire simulation time. This is also shown in Figure \ref{figs:comparison_one_d}, where the Hellinger distances associated with the GCQ scheme (FD-vs-Proj-GCQ-9) only oscillate between one and ten times the lowest Hellinger distance obtained from the static bijection with 18 quadrature nodes (FD-vs-static-18). In contrast to this and at the end of the simulation, the Hellinger distance between the finite difference solution and the projection filter approximation using the static bijection with 9 quadrature nodes (FD-vs-static-9) is about one hundred times greater than that from the lowest Hellinger distance (FD-vs-static-18). Thus, this simulation evidently shows that the parametric bijection \eqref{eq:phi_xi_gaussian} is superior to both the static and Gauss–Hermite bijections \eqref{eq:Gauss_Hermite_Bijection} for this example which portrays a hard stochastic nonlinear filtering problem to solve.

\begin{figure}
    \centering
    \includegraphics[trim={0cm 0.5cm 0.0cm 0.2cm},clip,width=0.9\linewidth]{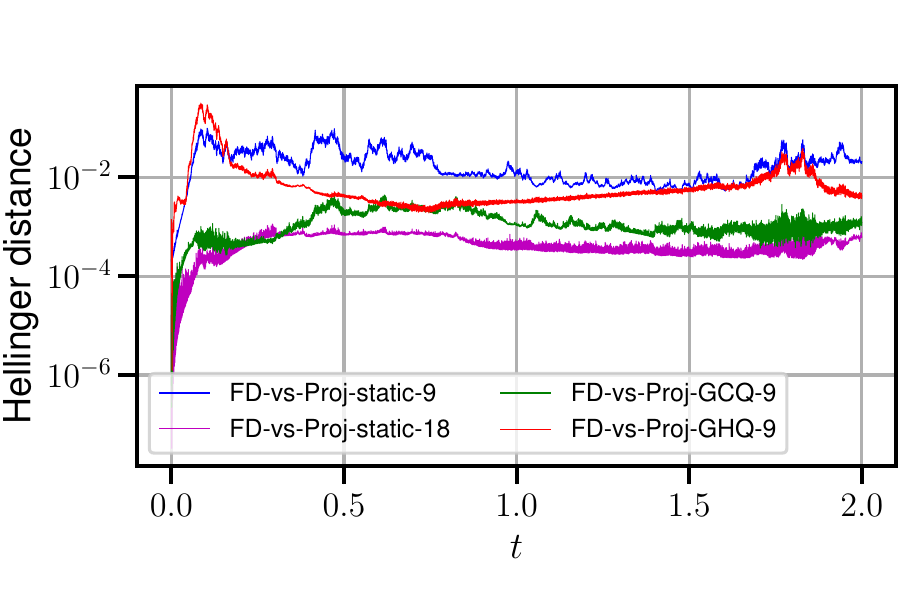}
    \caption{The Hellinger distances from the finite difference solutions to densities solved by projection filters with static bijections with $9$ and $18$ nodes (FD-vs-static-9 and FD-vs-static-18), or with parametric bijections \eqref{eq:phi_xi_gaussian} and \eqref{eq:Gauss_Hermite_Bijection} with $9$ nodes, (FD-vs-Proj-GCQ-9 and FD-vs-Proj-GHQ-9), respectively.}
    
    \label{figs:comparison_one_d}
    \end{figure}

\subsection{Modified Van der Pol Oscillator}

In this section, we compare the projection filter with a bootstrap particle filter with systematic resampling \cite{Chopin2020}. The dynamic model considered here is a modified Van der Pol oscillator, the standard form of which is widely used as a model for oscillatory processes in physics, electronics, biology, neurology, sociology and economics:
\begin{align} 
        d\mqty[x_{1,t}\\ x_{2,t}] &= \mqty[\kappa x_{1,t}+ x_{2,t} \\ -x_{1,t}+\kappa x_{2,t} + \mu (1 - x_{1,t}^2)x_{2,t}]dt +\mqty[0\\ \sigma_w] dW_t,\nonumber\\ 
        dy &= x_{1,t} dt + \sigma_v dV_t. \label{eqs:SDE_VDP_n_d_linear}
\end{align} 
In this simulation, we set $\mu = 0.3$, $\kappa=1.25$, and $\sigma_v=\sigma_w = 1$. We also set the simulation time step to be $2.5\times 10^{-3}$. Unlike the case with $\kappa=0$, where the probability density evolution can be easily contained using a compact support \cite{emzir2023}, the probability densities corresponding to \eqref{eqs:SDE_VDP_n_d_linear} with $\kappa >0$ expand quickly in time. We use here both GPQ and GHQ with their sparse-grid integration schemes where we set the level to four. For the GPQ scheme, the number of nodes used is 129 while for the GHQ scheme, it is 189 after ignoring all nodes with weights less than $10^{-9}$. For the particle filter, we use $9.6\times 10^6$ samples in our simulation.  We discretize the dynamic model \eqref{eqs:SDE_VDP_n_d_linear} using Euler--Maruyama for both the particle filter and the measurement process. For a multi index $\bunderline{i}\in\mathbb{N}^2$, define $x^{\bunderline{i}} = x_1^{i(1)}x_2^{i(2)}$. The natural statistics are set to be ${ x^{\bunderline{i}} }$, where $1 \leq \abs{\bunderline{i}}\leq 4$. Further, the initial density is set to be the standard Gaussian density.
To show the performance of the projection filter obtained by both sparse integration schemes, we calculate the empirical densities of the particle filter on a fixed grid. The comparison of the empirical densities from the particle filters with those from the projection filters is shown in Figure \ref{fig:particle_vs_projection_densities}, while the Hellinger distances between the empirical densities and those from the projection filters are in Figure \ref{fig:Hellinger_VDP}. Figure \ref{fig:particle_vs_projection_densities} shows that the bijected quadrature nodes of both sparse GPQ and GHQ schemes are systematically adapting to the shapes of the densities as they evolve in time. The shapes of the projection densities resemble those of the empirical densities from the particle filter, except that the projection filter's densities cannot capture some sharp notch-like shapes with low values as in Figure \ref{fig:particle_vs_projection_densities}. The Hellinger distances from the projection filter's densities obtained using the GHQ scheme to the empirical densities are slightly lower compared to those obtained using the GPQ scheme. 

\textcolor{black}{Table \ref{tab:execution_time_comparison} shows computational times of the projection filter with different sparse-grid levels and adaptive bijections, and the projection filter with a static bijection. For the static bijection, we use Gauss--Patterson sparse grid integration, where we set the static bijection to be $\tanh^{-1}$ following \cite{emzir2023}. As can be seen, the execution time of the projection filter using GHQ level 4 with bijection \eqref{eq:Gauss_Hermite_Bijection} is the fastest compared to the other projection filter schemes. We found that even with the maximum sparse-grid level available for Gauss--Patterson sparse grid integration (level 8), the projection filter with static bijection produces ill-defined projection densities around $t=0.7$. Therefore, to accurately implement the projection filter using the static bijection, a sparse-grid level higher than 8 is necessary, which means a significant increase in the quadrature nodes that will result in a substantial rise in execution time. As shown in Table \ref{tab:execution_time_comparison}, the computational times for the GHQ with bijection \eqref{eq:Gauss_Hermite_Bijection} increase only modestly, compared to those of the GPQ with bijection \eqref{eq:phi_xi_gaussian}. %
}

\newcommand{\lefttrimVDP}{5.0}
\newcommand{\righttrimVDP}{5.0}
\newcommand{\uppertrimVDP}{15.0}
\newcommand{\lowertrimVDP}{15.0}
\newcommand{\vdpsimulationid}{Expanded_VDP14_08_2023_12_08_14}

\begin{figure}
    \centering
    \includegraphics[trim={0cm 2.2cm 0cm 3.5cm},clip,width=0.9\linewidth]{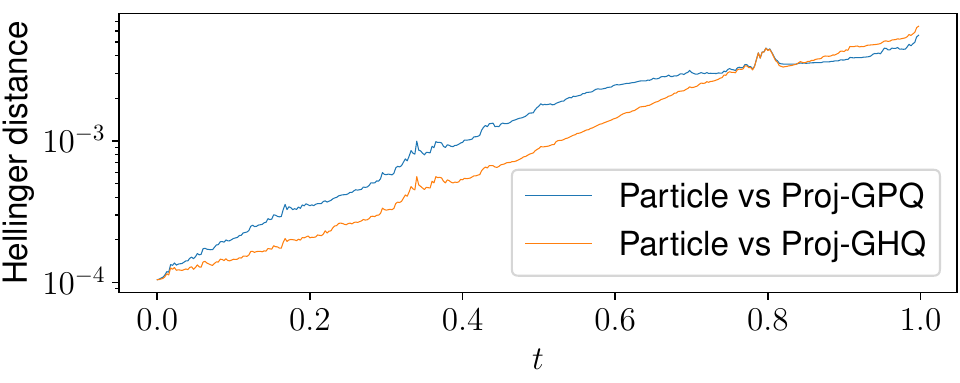}
    \caption{Hellinger distance from the empirical densities to the projection filter's densities solved using both GPQ and GHQ.}
    \label{fig:Hellinger_VDP}
\end{figure}

\begin{figure}
    \centering
    \makebox[\linewidth][c]{
    \includegraphics[trim={\lefttrimVDP cm \lowertrimVDP cm \righttrimVDP cm \uppertrimVDP cm},clip,width=0.9\linewidth]{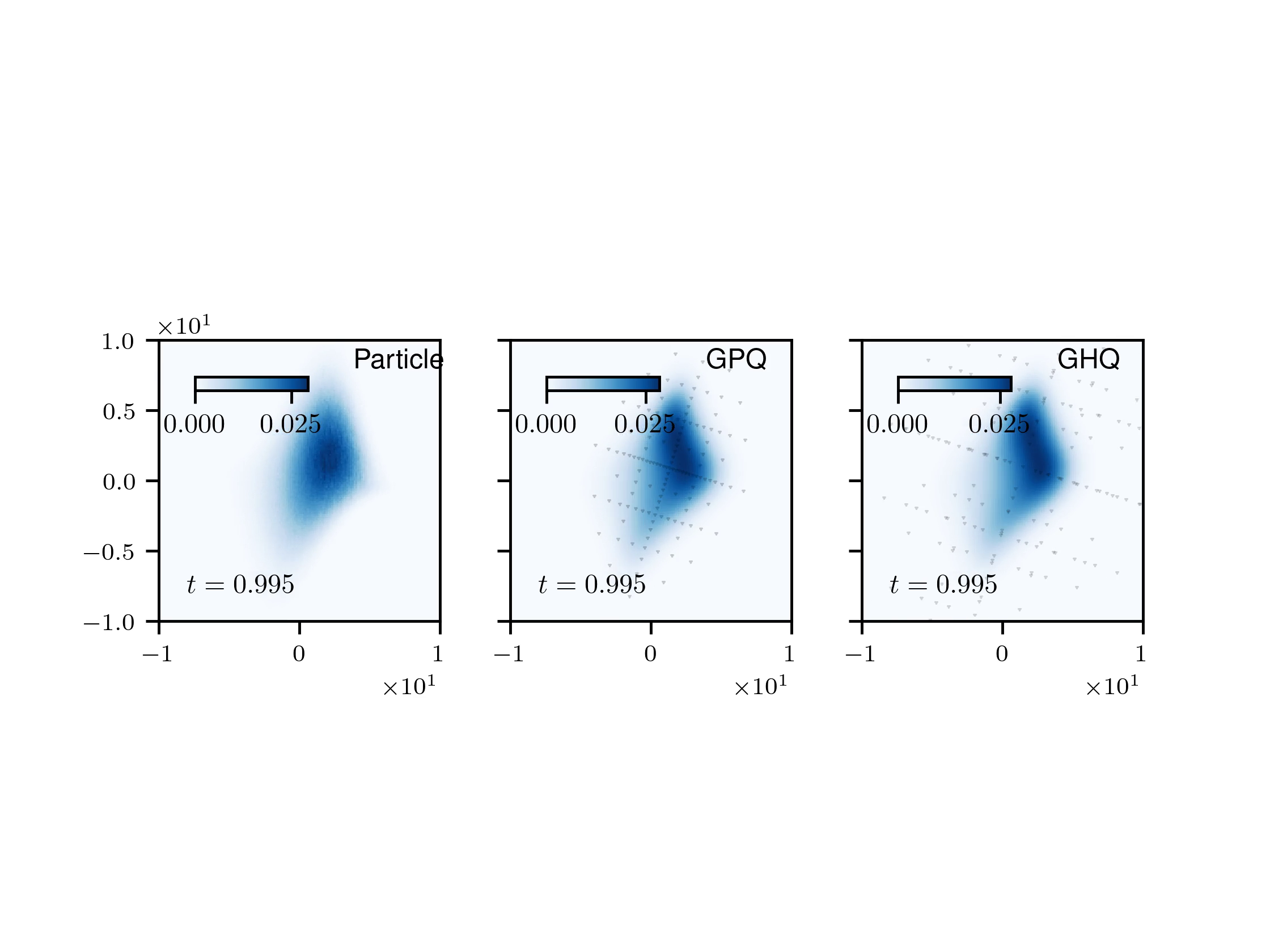}
    }
    \caption{Comparison of empirical densities from particle filter (left) and densities from the projection filters solved using GPQ (center) and GHQ (right) at $t=0.995$. The grey dots represents the position of the bijected quadrature nodes of the sparse Gauss--Patterson or Gauss--Hermite schemes, respectively.
    \label{fig:particle_vs_projection_densities}}
    \end{figure}

    \begin{table}[h]
    \caption{Computation time comparison for Section IV.B}
        \scriptsize
        \centering\begin{tabular}{| l | c | c | c | c | c |}
            \hline
                & \multicolumn{5}{c|}{Level}\\
            \hline
            Scheme & 4 & 5 & 6 & 7 & 8\\
            \hline
            GPQ & $3.543$s & $3.678$s  & $4.012$s & $4.785$s & $6.315$s\\
            \hline
            GHQ & $1.619$s & $1.636$s & $1.674$s & $1.833$s & $1.961$s \\
            \hline
            static-GPQ  & - & - & - & - & $1.964$s\\
            \hline
        \end{tabular}
        \label{tab:execution_time_comparison}
    \end{table}
    \newcommand{\lefttrimMoment}{0.45}
    \newcommand{\righttrimMoment}{0.25}
    \newcommand{\verticaltrimMoment}{0.35}
\subsection{Stochastic Epidemiology Application}

In this section, we consider an application of the projection filter to the stochastic suspected infected or recovery (SIR) nonlinear filtering problem, widely used for example in the study of the spread of infectious diseases. Consider the following SDE \cite{tornatore2005}:
\begin{align}
    d\mqty[x_{1,t}\\ x_{2,t}] &= \mqty[-\beta x_{1,t} x_{2,t} - \mu x_{1,t} + \mu \nonumber\\
    \beta x_{1,t}x_{2,t} - \left( \lambda + \mu \right) x_{2,t}] dt + \mqty[-\sigma x_{1,t}x_{2,t}\\ \sigma x_{1,t} x_{2,t}]dW_t,\\
        dy_{1,t} &= x_{2,t} dt + k dV_t.
\end{align}
In this equation, $x_{1,t}$ is the fraction of the suspected population, $x_{2,t}$ is the fraction of the infected population, assuming that the population size is constant. The fraction of the recovered population is given by $1-(x_{1,t}+x_{2,t})$, and its dynamic is non-stochastic which can be excluded from the SDE. The constants $\beta,\mu,\lambda$ correspond to the average number of contacts per infection per day, the birth rate, and the recovery rate of the infected people, respectively. The constants $\sigma$, and $k$ are positive constants associated with the process and measurement noises, respectively.

If $0<\beta<\min(\lambda+\mu -\frac{\sigma^2}{2},2\mu)$ then the disease-free equilibrium $x_\ast = (1,0)$ is globally asymptotically stable \cite[Theorem 2.1]{tornatore2005}. Therefore, the stationary probability measure has no density with respect to the Lebesgue measure. The high-density domain of the conditional probability density of SIR dynamics will be shrinking in time. In our simulation, we chose $\mu=0.2,\beta=0.14,\lambda=0.1,\sigma=0.2,k=10^{-4}$ and set the initial density to be a Gaussian density with mean $[0.95,0.02]^\top$ and variance $\text{diag}[0.95,0.02]\times 10^{-3}$. Using a similar exponential family to that of the previous section, we compare here the performance of the projection filter achieved by using the bijection \eqref{eq:phi_xi_gaussian} and GPQ level 5 with that of the particle filter. The Hellinger distance between the two densities at different times can be seen in Figure \ref{fig:Hellinger_SIR}. This, therefore, clearly shows that the bijection \eqref{eq:phi_xi_gaussian} can also perform very well in accurately tracking probability densities that shrink in time, as is the case here with SIR dynamics, and, by the same token, also demonstrates the adaptive capability of the proposed bijection \eqref{eq:phi_xi_gaussian}. \textcolor{black}{Compared to the projection filter with the proposed bijections, employing the projection filter alongside GPQ level 8 and the static bijection $\tanh^{-1}$ leads to an ill-defined projection density within just a few iterations.}
\newcommand{\sirsimulationid}{SIR_Filtering14_08_2023_18_47_14}

\begin{figure}[H]
    \centering
    \includegraphics[trim={0cm 2.2cm 0cm 3.5cm},clip,width=0.9\linewidth]{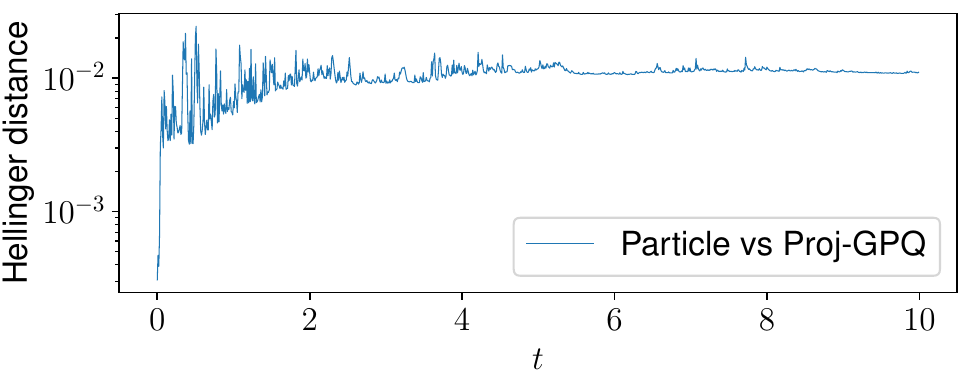}
    \caption{Hellinger distance from the empirical densities to the projection filter's densities solved using GPQ level 5.}
    \label{fig:Hellinger_SIR}
\end{figure}

\section{Conclusions}\label{sec:Conclusions}
In this work, we have introduced two new parametric bijections for the automatic projection filter that was recently proposed in \cite{emzir2023}. The first bijection was constructed by selecting a Gaussian density $q_\xi$ whose parameters were obtained by minimizing the KL divergence between the projected density $p_\theta$ and the Gaussian density $q_\xi$, which is equivalent to evaluating moment-matching conditions. We have shown that this bijection also minimizes the squared integration error under some sufficient theoretical conditions. The second bijection was also constructed via the same moment-matching conditions, but it was tailored for the GHQ scheme. We then applied these bijections to three practically-motivated numerical examples, and we found that they all achieved superior performance in terms of the Hellinger distances to the ground truth than the static bijection, using fewer quadrature nodes and thus achieving both higher accuracy and reduced computational time.

\bibliographystyle{IEEEtran}
\bibliography{Zotero_BibTeX}


\section*{Supplementary material}

# Supplementary Materials for *Gaussian-Based Parametric Bijections For Automatic Projection Filters*

Muhammad Emzir, Zheng Zhao, Lahouari Cheded, Simo Särkkä

## I. Introduction

In this supplementary documentation, in Table I, we give the essential notations employed in the main paper. Alongside this, we offer supplementary plots that augment the paper's content.

## II. Notations

### TABLE I
### List of important notations used in the paper.

| Variable | Description |
|---|---|
| $d$ | Dimension of the stochastic dynamics' state. |
| $d_y$ | Dimension of the measurement vector. |
| $\mathcal{D}$ | Canonical hypercube, $(-1, 1)^d$. |
| erf | Error-function. |
| $\mathbb{E}$ | Expectation operator. |
| $\mathbb{E}_{\theta, N}[f; \xi]$ | Approximated expectation of an arbitrary function $f$ under a parametric density $p_\theta$ via the numerical quadrature $Q_N^{d, \omega}$. |
| EM$(c)$ | Set of all exponential density functions with natural statistics $c$, $\{p \in \mathcal{P} : p(x) = \exp(c(x)^\top \theta - \psi(\theta))\}$. |
| KL$(p||q)$ | Kullback-Leibler divergence between two densities, $p$ and $q$. |
| $\mathcal{L}$ | Backward diffusion operator. |
| $m$ | Number of natural parameters. |
| $\mathcal{P}$ | Class of probability density functions with respect to the standard Lebesque measure on $\mathcal{X}$. |
| $Q_N^{d, \omega} f$ | $d$-dimensional numerical quadrature of a function $f$ with a positive weight function $\omega(\tilde{x})$ and $N$ quadrature nodes. |
| $Q_t, R_t$ | Invertible spectral density matrices of $W_t$ and $V_t$, at $t \geq 0$, respectively. |
| $\mathcal{S}$ | Class of parametric densities. |
| $W_t, V_t$ | Independent Wiener processes taking values in $\mathbb{R}^{d_w}$ and $\mathbb{R}^{d_y}$, respectively. |
| $\mathcal{W}_d^r(\mathcal{D})$ | Space of functions with bounded mixed partial derivatives up to order $s$; i.e., $\mathcal{W}_d^r(\mathcal{D}) := \{f : \mathcal{D} \to \mathbb{R} : \|\frac{\partial^{|s|} f}{\partial \tilde{x}_{(1)}^{s_1} \dots \partial \tilde{x}_{(d)}^{s_d}}\|_\infty < \infty, s_i < r\}$. |
| $x$ | State of the dynamics. |
| $\mathcal{X}$ | Space of all possible states. |
| $\mathcal{Y}$ | Space of all possible measurements. |
| $\theta$ | Natural parameters. |
| $\Theta$ | Set of all natural parameters. |
| $\phi$ | Fixed smooth bijection $\mathcal{D} \to \mathcal{X}$. |
| $\phi_\xi$ | Parametrized smooth bijection $\mathcal{D} \to \mathcal{X}$ with parameters $\xi$. |
| $\psi(\theta)$ | Cumulant-generating function at $\theta$ w.r.t EM$(c)$. |

## III. Additional Plots

### A. Univariate Problem in Section IV.A

Figure 1 illustrates a comparison between conditional densities derived from solving Kushner–Stratonovich SPDEs using the finite difference method (considered the ground truth) and those computed using the projection filter at different time points. This figure, similar to the Hellinger distance plot (Figure 2 in the paper), displays the similarity between the conditional densities approximated by the projection filter, utilizing Gauss–Chebyshev numerical integration and bijection (20) (Proj-GCQ-9), and the finite difference solutions to the Kushner–Stratonovich equation (FD). Qualitatively, the densities Proj-GCQ-9 are comparable to the approximated conditional densities generated by the projection filter with a static bijection and twice the quadrature nodes (Proj-Static-18).

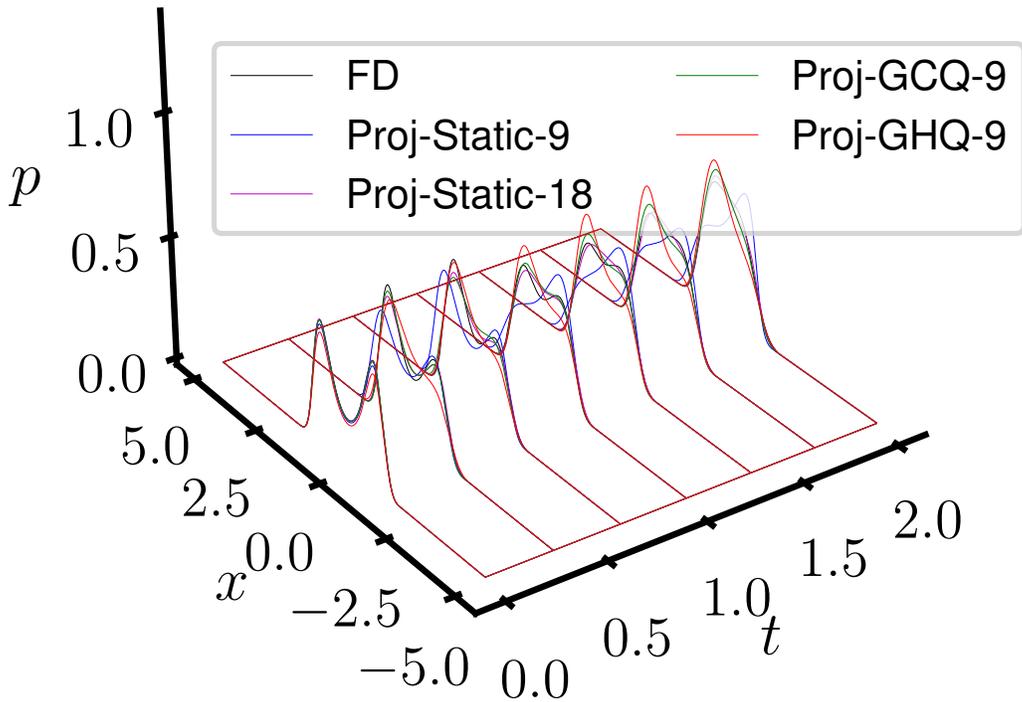

Fig. 1. This figure shows a comparison between the densities obtained by employing finite difference schemes (represented in black), densities obtained through the application of projection filters using a static bijection with 9 and 18 quadrature nodes and densities obtained from solving projection filters utilizing the bijections provided in formulas (20) and (24a) (depicted in blue and magenta respectively).

*B. Bivariate Problem in Section IV.B*

Figure 2 shows that even with the maximum sparse-grid level available for Gauss–Patterson sparse grid integration (level 8), the projection filter with static bijection produces an ill-defined projection densities around $t = 0.7$, which is reflected in the explosive Hellinger distance between the empirical densities from the particle filter and the densities obtained via the projection filter with the static bijection. Therefore, to accurately implement the projection filter using the static bijection, a sparse-grid level higher than 8 is necessary, which means a significant increase in the quadrature nodes that will result in a substantial rise in execution time.

Figure 4 shows a comparison of moments calculated through partial differentiation of the cumulant-generating function using the projection filter algorithm and those derived from the particle filters, which are considered the ground truth. Notably, the moments obtained from GHQ exhibit a closer correspondence with the particle filter's moments than those derived from GPQ. These plots also show corresponding polynomial functions of state realizations for illustrative purposes.

*C. Bivariate Problem in Section IV.C*

The computational times for the numerical example in Section IV.C exhibit similarity to those in Section IV.B, as both involve bivariate dynamics. It is, however, crucial to highlight that in the numerical example outlined in Section IV.C, employing the projection filter alongside GPQ level 8 and the $\tanh^{-1}$ static bijection leads to an ill-defined projection density within a few iterations (see Figure 3). This observation implies that for this specific optimal filtering problem, the usage of Gauss–Patterson sparse-grid quadrature with the static bijection $\tanh^{-1}$ demands a higher level of sparse-grid integration compared to that in Section IV.B. Consequently, this level would result in an increase in computational time. This also emphasizes the efficiency of the two proposed bijections in this paper, as employing these bijections along with sparse-grid quadrature enables seamless operation of the projection filter.

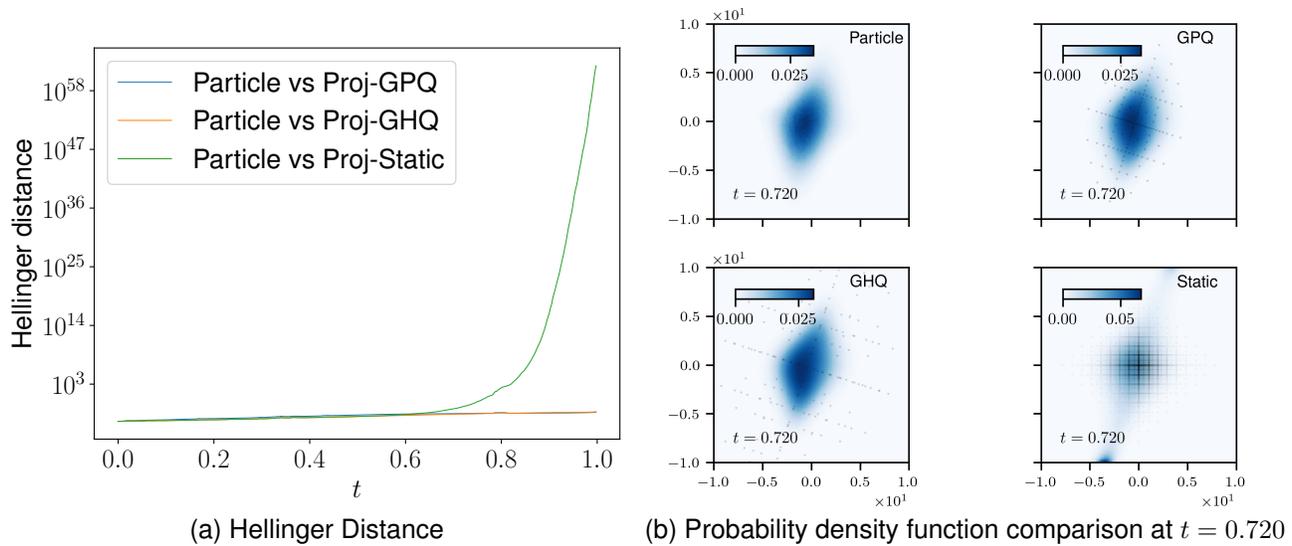

(a) Hellinger Distance

(b) Probability density function comparison at $t = 0.720$

Fig. 2. Figure 2a shows a comparison of Hellinger distances between the empirical densities of the particle filter and the projection filter densities obtained through adaptive and static bijections. In Figure 2b, it is obvious that starting from $t = 0.72$, the projection densities calculated using the static bijection $\tanh^{-1}$ and Gauss–Patterson sparse-grid quadrature level 8 begin to deviate significantly from the empirical densities of the particle filter. The deviations appear since significant portions of high-density regions moved outside the quadrature node locations.

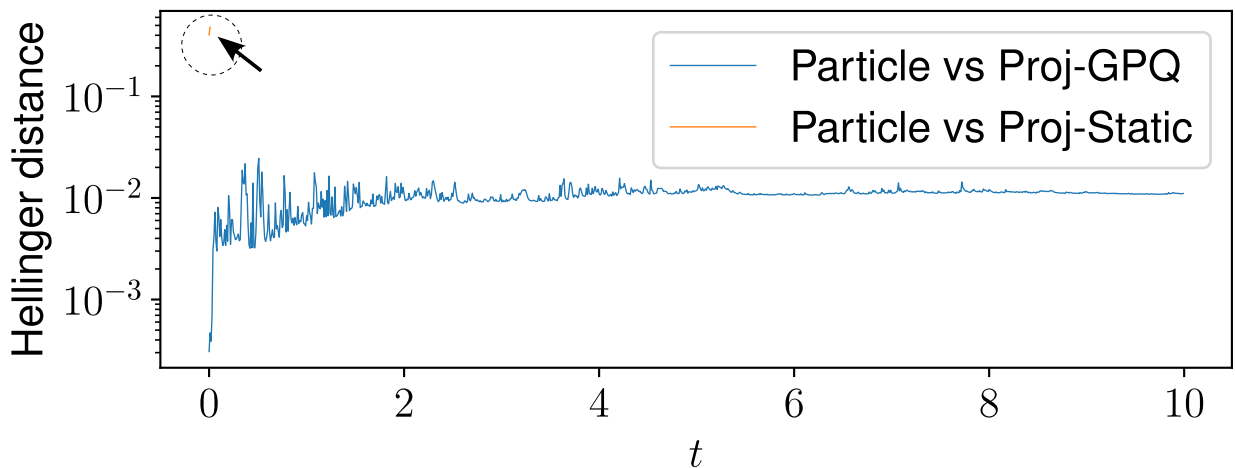

Fig. 3. Hellinger distance from the empirical densities to the projection filter's densities solved using GPQ level 5 with bijection (20), and GPQ level 8 with $\tanh^{-1}$ static bijection. The projection filter with the static bijection leads to an ill-defined projection density within just a few iterations (see the black arrow).

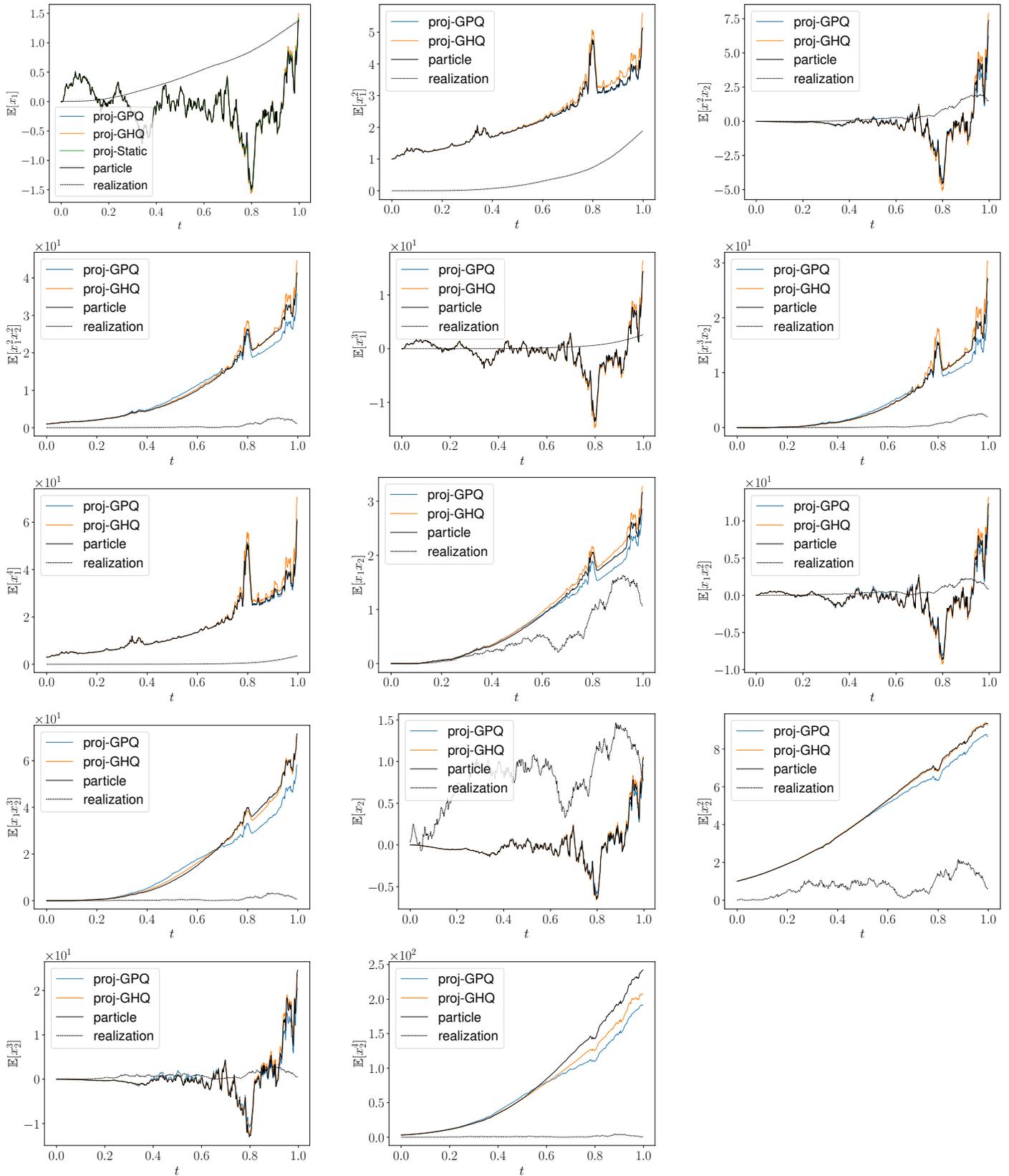

Fig. 4. Comparison approximation of natural statistic's expected values solved by the projection filter and the particle filter.



\end{document}